\documentclass[twoside,a4paper]{amsart}

\usepackage{pdfsync}

\usepackage{xcolor}
\definecolor{webgreen}{rgb}{0,.5,0}
\definecolor{webbrown}{rgb}{.6,0,0}
\definecolor{RoyalBlue}{cmyk}{1, 0.50, 0, 0}
\usepackage[colorlinks=true, breaklinks=true, urlcolor=webbrown, linkcolor=RoyalBlue, citecolor=webgreen,backref=page]{hyperref}

\usepackage{courier} 
\normalfont
\usepackage{epsfig, graphicx, subfigure}
\usepackage{verbatim, setspace}
\usepackage{amsmath, amssymb}

\newcommand{\R}     {\mathbb{R}}
\newcommand{\N}     {\mathbb{N}}
\newcommand{\Z}     {\mathbb{Z}}
\newcommand{\cws}{\stackrel{*}{\to}}

\newcommand{\supp}{\mathrm{supp}}
\newcommand{\diag}{\mathrm{diag}}

\renewcommand{\arg}{\mathrm{arg}}
\renewcommand{\det}{\mathrm{det}}

\newcommand{\dd}{\mathrm{d}}
\newcommand{\ic}{\mathrm{i}}
\newcommand{\RS}{\boldsymbol{\mathfrak{R}}}

\newcommand{\z} {{\boldsymbol z}}

\newcommand{\n} {{\vec n}}
\newcommand{\vc}    {{\vec c}}

\def\cal{\mathcal}
\let\Re=\undefined
\DeclareMathOperator{\Re}{Re}
\let\Im=\undefined
\DeclareMathOperator{\Im}{Im}

\def\ge{\geqslant}
\def\le{\leqslant}

\newtheorem{theorem}{Theorem}[section]
\newtheorem{proposition}[theorem]{Proposition}
\newtheorem{corollary}[theorem]{Corollary}
\newtheorem{lemma}[theorem]{Lemma}

\newtheorem{definition}[theorem]{Definition}

\theoremstyle{remark}
\newtheorem*{remark}{Remark}

\numberwithin{equation}{section}

\begin{document}

\title[Multiple orthogonal polynomials and Jacobi matrices on trees]{Self-adjoint Jacobi matrices on trees and \\ multiple orthogonal polynomials }

\author[A.I. Aptekarev]{Alexander I. Aptekarev}
\address{Keldysh Institute of Applied Mathematics, Russian Academy of Science, Moscow, Russian Federation}
\email{\href{mailto:aptekaa@keldysh.ru}{aptekaa@keldysh.ru}}

\thanks{The research of the first author was carried out with support from a grant of the Russian Science Foundation (project RScF-14-21-00025).
The work of the second author done in the last section
 of the paper was supported by a grant of the Russian Science Foundation (project RScF-14-21-00025)
 and his research
on the rest of the paper was supported by the grant NSF-DMS-1464479 and Van Vleck Professorship Research Award.
The research of the third author was supported in part by a grant from the Simons Foundation, CGM-354538.}

\author[S. Denisov]{Sergey A. Denisov}
\address{Department of Mathematics, University of Wisconsin-Madison, 480 Lincoln Dr., Madison, WI 53706, USA}
\email{\href{mailto:denissov@math.wisc.edu}{denissov@math.wisc.edu}}

\author[M. Yattselev]{Maxim L. Yattselev}
\address{Department of Mathematical Sciences, Indiana University-Purdue University Indianapolis, 402~North Blackford Street, Indianapolis, IN 46202, USA}
\email{\href{mailto:maxyatts@math.iupui.edu}{maxyatts@math.iupui.edu}}

\begin{abstract}
We consider a set of measures on the real line and the
corresponding system of multiple orthogonal polynomials (MOPs) of the
first and second type. Under some very mild assumptions, which are
satisfied by Angelesco systems, we define self-adjoint Jacobi
matrices on certain rooted trees. We express their Green's functions and
the matrix elements in terms of MOPs. This provides a generalization of the well-known connection between the theory of polynomials orthogonal on the real line and Jacobi matrices on $\mathbb{Z}_+$  to higher dimension. We illustrate importance of this connection by proving ratio asymptotics for MOPs using methods of operator theory.
\end{abstract}



\maketitle

\setcounter{tocdepth}{3}
\tableofcontents

\section{Introduction}

The theory of polynomials orthogonal on the real line is known to
play an important role in the spectral theory of Jacobi matrices. In
this paper, we show that the theory of multiple orthogonal
polynomials (MOPs) is related to the spectral theory of Jacobi
matrices on rooted trees. We will start this introduction by
recalling definition and main properties of MOPs.

\subsection{Multiple orthogonal polynomials}
\label{S1.1}

In what follows we shall set $\N:=\{1,2,\ldots\}$ and $\Z_+:=\{0,1,2\ldots\}$. Consider a vector $$\vec{\mu}:=(\mu_1,\ldots,\mu_d), \quad d\in\N,$$  of positive finite Borel measures defined on $\mathbb{R}$ and let $$\vec{n}:=(n_1,\ldots,n_d)\in\mathbb{Z}_+^d,\quad |\vec{n}|:=\sum\limits_{j=1}^dn_j.$$
In this paper, we always assume that  $\supp \,\mu_j$ is not a finite set of points and that
\[
\int_{\mathbb{R}} x^ld\mu_j(x)<\infty
\]
for every $j\in \{1,\ldots,d\}$ and every $l\in \Z_+$.

\begin{definition}
Polynomials $\big\{A_{\vec{n}}^{(j)}\big\}_{j=1}^d$ that satisfy
\begin{eqnarray}
\deg A_n^{(j)}\leqslant n_j-1\quad\nonumber  {\rm\it for\,\, all\,\,} j\in \{1,\ldots,d\} \\
\int_{\mathbb{R}}\sum\limits_{j=1}^d A_{\vec{n}}^{(j)}(x)x^{l}d\mu_j(x)=0
\label{1.1} \quad  {\rm\it for\,\, all\,\,} l\in \{0,\ldots, |\vec{n}|-2\}
\end{eqnarray}
are called type I multiple orthogonal polynomials.
\end{definition}

\begin{remark}
In the definition above, we let $A_n^{(j)}=0$ if $n_j-1<0$.
\end{remark}

\begin{definition}
Polynomial $P_{\vec{n}}$ is called type II multiple orthogonal polynomial if it  satisfies
\begin{eqnarray}
 \deg P_{\vec{n}}\leqslant|\vec{n}|,\nonumber\\
 \label{1.2}
\int_{\mathbb{R}} P_{\vec{n}}(x)x^ld\mu_j(x)=0
\quad  {\rm\it for\,\, all\,\,} j\in \{1,\ldots,d\}\,\,{\rm \it and\,\, all} \,\,  l\in \{0,\ldots, n_j-1\}\,.
\end{eqnarray}
\end{definition}
Orthogonality relations  \eqref{1.1} and \eqref{1.2} define enough linear homogeneous equations to determine the coefficients of $A_\n^{(j)}$ and $P_\n$. Thus, polynomials of the first and second type always exist. The question of uniqueness is more involved. If $ P_\n$ is defined uniquely up to a constant, then the multi-index $\vec{n}$ is called normal and we choose the following normalization
\[
 P_{\vec{n}}(x)=x^{|\vec{n}|}+\cdots\,,
\]
i.e., the polynomial $ P_{\vec{n}}$ is monic. It turns out that $\vec{n}$ is normal  if and only if the following linear form
\begin{equation}\label{sad1}
Q_{\vec{n}}(x):=\sum\limits_{j=1}^d A_{\vec{n}}^{(j)}(x)d\mu_j(x)
\end{equation}
is defined uniquely up to multiplication by a constant. In this case, we will normalize the polynomials of the first type by
\begin{equation}\label{n_2}
\int_{\mathbb{R}} x^{|\vec{n}|-1}Q_{\vec{n}}(x)=1\,.
\end{equation}
Following Mahler \cite{Mah68}, we shall say that
\begin{definition}
The vector $\vec{\mu}$  is called perfect if all the multi-indices $\vec{n}\in \mathbb{Z}_+^d$ are normal.
\end{definition}

Together with the multiple orthogonal polynomials we shall also need their functions of the second kind.
\begin{definition}
\label{def:1.3}
Functions \( \big\{ R_\n^{(j)}\big\} \) defined by
\begin{equation}
\label{Rnj}
R_\n^{(j)}(z) :=  \int\frac{P_\n(x)}{z-x}d\mu_j(x), \quad j\in\{1,\ldots,d\},
\end{equation}
are called the functions of the second kind associated to the polynomial \( P_\n \). Similarly,
\begin{equation}
\label{Ln}
L_\n(z) := \int_{\mathbb{R}}\frac{Q_{\vec{n}}(x)}{z-x}
\end{equation}
is the function of second kind associated to the linear form \( Q_\n \).
\end{definition}

Given a measure $\mu$ on the real line, denote by \( \widehat\mu \) the following Cauchy-type integral
\begin{equation}
\label{markov}
\widehat{\mu}(z):=\int_{\mathbb{R}}\frac{d\mu(x)}{z-x}\,, \quad z\not\in\supp\,\mu\,,
\end{equation}
which, following the initial work of Markov \cite{Mar95}, is often referred to as a Markov function. Then, it follows from Definition~\ref{def:1.3} and orthogonality relations \eqref{1.2} that polynomials
\begin{equation}
\label{Pnj}
P_\n^{(j)}(z) := \int_{\mathbb{R}} \frac{P_\n(z)-P_\n(x)}{z-x}d\mu_j(x)\,, \quad j\in\{1,\ldots,d\}\,,
\end{equation}
satisfy
\[
R_\n^{(j)}(z) = P_\n(z)\widehat\mu_j(z) - P_\n^{(j)}(z) = \mathcal O\big(z^{-n_j-1}\big)\,, \quad j\in\{1,\ldots,d\}\,,
\]
where \( \mathcal O(\cdot) \) holds as \( z\to\infty \). Thus, to each vector of Markov functions \( (\widehat\mu_1,\ldots,\widehat\mu_d) \), type II multiple orthogonal polynomials allow us to define a vector of rational approximants \( (P_\n^{(1)}/P_\n,\ldots,P_\n^{(d)}/P_\n) \). Similarly, the polynomial
\[
A_\n^{(0)}(z) :=\int_{\mathbb{R}} \frac{Q_\n(z)-Q_\n(x)}{z-x}= \sum_{j=1}^d\int_{\mathbb{R}}\frac{A_\n^{(j)}(z)-A_\n^{(j)}(x)}{z-x}d\mu_j(x)
\]
satisfies
\begin{equation}
\label{LnInfty}
L_\n(z) = \sum_{j=1}^d A_\n^{(j)}(z)\widehat\mu_j(z) - A_\n^{(0)}(z) = z^{-|\n|} + \mathcal O\big(z^{-|\n|-1}\big)\,,
\end{equation}
where again \( \mathcal O(\cdot) \) holds as \( z\to\infty \). Hence, to each vector of Markov functions \( (\widehat\mu_1,\ldots,\widehat\mu_d) \), type I multiple orthogonal polynomials allow us to define a linear form that approximates this vector.

Multiple orthogonal polynomials and the corresponding approximants were introduced by Hermite in \cite{Herm73} as the main tool in his famous proof of the transcendency of \( e \). Later, Pad\'e undertook a systematic study of the case \( d=1 \) \cite{Pade92} (in this case both types of polynomials coincide up to an index shift and normalization). Nowadays, MOPs  and the corresponding approximants are often referred to as Hermite-Pad\'e polynomials and Hermite-Pad\'e approximants. For more information about multiple orthogonal polynomials, we refer the reader to survey papers \cite{Nut84,AS_PAS92,Ap98} and monograph~\cite{NikishinSorokin}.  For some recent results in the theory of MOPs, we refer the reader to \cite{vv5,MR3137137,vv3,vv7,vv4,vv1,MR3489559,vv2,vv6,MR3687129,MR3807896}.

\subsection{Lattice recurrence relations}
\label{S1.1b}

MOPs satisfy various recurrences (see, e.g.,  \cite{VA11,ApKal_Pal98,ApKalLLRocha06}).  We will be interested in the relationship between the nearest neighbors on the lattice $\vec{n}\in \mathbb{Z}_+^d$, where $\vec{n}$ is the index of orthogonal polynomial. Henceforth, we denote by  $\vec{e}_1:=(1,0,\ldots,0), \ldots, \vec{e}_d:=(0,\ldots,0,1)$ the standard basis in \( \R^d \).
For the linear forms $\{Q_{\vec{n}}\}$, we have (see, e.g., \cite{Ismail,VA11})
\begin{equation}\label{1.5}
xQ_{\vec{n}}(x)=Q_{\vec{n}-\vec{e}_j}(x)+b_{\vec{n}-\vec{e}_j,j}Q_{\vec{n}}(x)+\sum\limits_{l=1}^d a_{\vec{n},l}
Q_{\vec{n}+\vec{e}_l}(x)\,,\quad j\in\{1,\ldots,d\}, \quad \n \in \N^d,
\end{equation}
and type II polynomials satisfy
\begin{equation}\label{1.7}
xP_{\vec{n}}(x)=P_{\vec{n}+\vec{e}_j}(x)+b_{\vec{n},j}P_{\vec{n}}(x)+\sum\limits_{l=1}^d a_{\vec{n},l}
P_{\vec{n}-\vec{e}_l}(x)\,,\quad j\in\{1,\ldots,d\}, \quad \vec{n}\in \mathbb{Z}_+^d.
\end{equation}
In this equation, we let $P_{\vec{n}-\vec{e}_l}=0$ and $a_{\vec{n},l}=0$ if at least one of the components in the vector $\vec{n}-\vec{e}_l$ is negative.

It is known that the real-valued parameters  $\{a_{\vec{n},j}\}$ and  $\{b_{\vec{n},j}\}$ are uniquely determined by
 $\vec{\mu}$ (see formulas \eqref{por1} and \eqref{por} from Appendix~\ref{appA}).  From the definition of the polynomials of the second type, it is clear that, e.g., $\{P_{n\vec{e}_j}\}, n\in \mathbb{Z}_+$, are
monic polynomials orthogonal on the real line with respect to a
single measure $\mu_j$ and, when written for $\vec{n}=n\vec{e}_j$,
 exactly one of the equations \eqref{1.7} represents the
standard three term recurrence which will be discussed later. In general, setting some of the indices in $\vec{n}=(n_1,\ldots,n_d)$ to zero, e.g., letting $\vec{n}=(n_1,\ldots,n_l,0,\ldots,0)$ reduces the system to the one defined by truncated vector $(\mu_1,\ldots,\mu_l)$ and the corresponding recursions on the boundary can be viewed as lower-dimensional recursions.

If $d=1$, type II polynomials $\{P_n\}$ are the standard monic polynomials orthogonal on the real line with respect to the measure $\mu_1$ and
\[
A_n^{(1)}=\frac{P_{n-1}}{\|P_{n-1}\|_{\mu_1}^2}\,, \quad n\in \mathbb{N}\,.
\]
Equations \eqref{1.7} specialize to the standard three-term recurrence
\begin{equation}\label{sd_1}
xP_n(x)=P_{n+1}(x)+b_{n,1}P_{n}(x)+ a_{n,1}P_{n-1}(x)\,.
\end{equation}

Later in the paper, when $d=1$, will write $\mu$, $a_{n-1}$, $b_n$ instead of  $\mu_1$, $a_{n,1}$, $b_{n,1}$. It is known that $a_n>0, b_n\in \mathbb{R}$ for all $n\in \Z_+$ and, if $\mu$ is compactly supported, then
\begin{equation}
\sup_{n} a_n<\infty, \quad \sup_{n}|b_n|<\infty
\end{equation}
as follows from \eqref{1.19}, \eqref{lope}, and \eqref{lope1}  below.

Coefficients $\{a_n\}$ and $\{b_n\}$ define a one-sided tri-diagonal operator $\cal{H}$ that can be symmetrized to get a self-adjoint bounded operator $\cal{J}$, i.e., the Jacobi matrix,  (see formulas \eqref{1.14} and \eqref{1.20} below). Conversely, we can start with arbitrary $\{a_n\}, \{b_n\}$ that satisfy \[a_n>0,\quad  \sup_{n}a_n<\infty, \quad\sup_n|b_n|<\infty\] and define a self-adjoint bounded Jacobi matrix $\cal{J}$. Polynomials $\{P_n\}$ are determined by solving recursion \eqref{sd_1} with initial conditions  $P_{-1}=0,\,P_0=1$. Then, one can show that there exists a unique  measure $\mu$  for which $\{P_n\}$ are monic orthogonal. This $\mu$ turns out to be compactly supported.

 If $d>1$, unlike the one-dimensional case, we can not prescribe $\{a_{\vec{n},j}\}$ and $\{b_{\vec{n},j}\}$ arbitrarily. In fact, coefficients in \eqref{1.5} and \eqref{1.7}  satisfy the so-called ``consistency conditions'' which is a system of nonlinear difference equations (see, e.g., Theorem 3.2 in \cite{VA11}):
 \begin{eqnarray}\label{sd_i1}
 b_{\vec{n}+\vec{e}_i,j}- b_{\vec{n},j}=b_{\vec{n}+\vec{e}_j,i}- b_{\vec{n},i},\\ \label{sd_i2}
 \sum_{k=1}^d a_{\vec{n}+\vec{e}_j,k} -\sum_{k=1}^d a_{\vec{n}+\vec{e}_i,k}=b_{\vec{n}+\vec{e}_j,i}b_{\vec{n},j}-b_{\vec{n}+\vec{e}_i,j}b_{\vec{n},i},\\
 a_{\vec{n},i}(b_{\vec{n},j}-b_{\vec{n},i})=a_{\vec{n}+\vec{e}_j,i}(b_{\vec{n}-\vec{e}_i,j}-b_{\vec{n}-\vec{e}_i,i}),\label{sd_i3}
 \end{eqnarray}
 where $\vec{n}\in \N^d$ and $i,j\in \{1,\ldots,d\}$.  Relations \eqref{sd_i1}--\eqref{sd_i3} can be viewed as a
  discrete integrable system (see, e.g.,  \cite{Bob_DIS04}) whose associated Lax pair was studied in  \cite{ApDerVA16}.  
  
\subsection{Angelesco systems}

 In the one-dimensional case, recurrence relations \eqref{sd_1} establish a connection between the theory of orthogonal polynomials and the spectral theory of Jacobi matrices \cite{NikishinSorokin}. Therefore, it is natural to ask what self-adjoint operators are related to multidimensional equations \eqref{1.5} and \eqref{1.7}? There were several results in this direction. In \cite{ApDerVA15,ApDerMikiVA16}, equations \eqref{1.5} and \eqref{1.7} were combined to obtain the electro-magnetic Schr\"{o}dinger operator defined on $\ell^2(\mathbb{Z}_+^d)$. These operators  were  symmetrized but only in very special cases.
  In \cite{Kal94,ApKal_Pal98,ApLLRocha05,ApKalLLRocha06, ApKalS09, Ap14}, the recurrences along the diagonal (the so-called ``step-line'') were related to higher-order difference relations on $\mathbb{Z}_+$. They were not self-adjoint, in general.  

The main goal of this paper is to introduce bounded self-adjoint operators defined on $\ell^2(\cal{T})$, where $\cal{T}$ is a tree (finite or infinite) for which $\{P_{\vec{n}}\}$ and $\{Q_{\vec{n}}\}$ turn out to be the generalized eigenfunctions  after  suitable normalization.  This will be done under the following assumptions on $\vec{\mu}$ and $\{a_{\vec{n},j}\},  \{b_{\vec{n},j}\}$:
\begin{equation}\label{1.9}
\left\lbrace
\begin{array}{l}
{\rm (A)}\; \vec{\mu}\;-\;\mbox{perfect}\,,\medskip \\
{\rm (B)}\; 0<a_{\vec{n},j}\,\,{\rm for\,\, all}\,\, \n\in \Z_+^d \text{ such that } n_j>0, \,\,j\in \{1,\ldots,d\}\,,\medskip\\
{\rm (C)}\;\,\sup \limits_{\vec{n}\in \mathbb{N}^d,j\in \{1,\ldots,d\}}a_{\vec{n},j}<\infty\,,\,  \sup \limits_{\vec{n} \in \mathbb{N}^d,j\in \{1,\ldots,d\}}|b_{\vec{n},j}|<\infty  \, .
\end{array}
\right.
\end{equation}
We will show that conditions \eqref{1.9} are satisfied by Angelesco systems which is defined as follows.

\begin{definition}
We say that $\vec\mu$ is an Angelesco system of measures if
\begin{equation}\label{1.10}
\Delta_i\cap\Delta_j=\varnothing,\quad i,j\in\{1,\ldots,d\},
\end{equation}
where $\Delta_i:= \mathrm{Ch}(\supp\, \mu_i)$ and  \( \mathrm{Ch}(\cdot) \) stands for the convex hull. We note here that, $\{\Delta_i\}$ is the system of $d$ closed segments separated by $d-1$ nonempty open intervals.   Without loss of generality, we can assume that $\Delta_1<\ldots<\Delta_d$ ($E_1<E_2$ if \( \sup E_1 <\inf  E_2 \)).
\end{definition}

Angelesco systems, being important in theory of Hermite-Pad\'e approximation and in other areas of analysis and number theory, were studied in numerous papers, see, e.g.,  \cite{Ang19,Nik79,GRakh81,Apt_Szego,Y16} and references therein.

The theory of Schr\"{o}dinger operators on graphs has been an active
 topic lately which was motivated  by their applications in the study of some  problems in
mathematical physics \cite{aizenman,froese,klein}, most notably the delocalization in Anderson model. For the general
 spectral theory of operators on trees and more
references, see, e.g., \cite{keller}. We believe that our paper will
set the ground for further development in the theory of MOPs and
spectral theory of difference operators on graphs. Among the problems for future research in this direction we mention
the problem of finding the spectrum and the spectral type of the Jacobi matrices on
the trees generated by MOPs and building the spectral theory for
Nikishin system of MOPs (see, e.g., \cite{Nik80,NikishinSorokin,gu} for definition of Nikishin system and recent developments).
Multiple orthogonal polynomials for some classical weights were recently studied in, e.g., \cite{walt1} and the recurrence coefficients were found explicitly. These results allow one to write the Jacobi matrix on the tree in the exact form. We are planning to study them in subsequent publications.

In the next section, we recall the classical connection between Jacobi matrices and orthogonal polynomials. In section~\ref{sec:JMT} we introduce Jacobi matrices on trees and explain their relationship to the theory of MOPs. Then, in  section~\ref{sec:JMAS}, we explore the fact that  Angelesco systems satisfy assumption \eqref{1.9}. In particular, we show how results on ratio asymptotics for MOPs can be obtained using the established connection between MOPs and Jacobi matrices.  Appendix~\ref{appA} contains the proof that Angelesco systems satisfy \eqref{1.9}  and some general results.  In Appendix~\ref{ApB}, we apply matrix Riemann-Hilbert problem technique to prove the asymptotics of the recurrence coefficients and MOPs for Angelesco system with analytic weights that is also used in section~\ref{sec:JMAS}.

\section{Classical Jacobi matrices}
\label{sec:CJM}

In this section we quickly review the connection between orthogonal polynomials and the spectral theory of Jacobi matrices. Hereafter, we adopt the following notation:
\begin{itemize}
\item If $\mu$ is a measure on $\mathbb{R}$, then we set
\[
\langle f,g\rangle_\mu:=\int_{\mathbb{R}} f\overline gd\mu\,,\quad \|f\|_\mu:=\langle f,f\rangle_\mu^{\frac 12}\,,\quad  \|\mu\|:=\int_{\mathbb{R}}d\mu\,.
\]

\item Let $\cal{G}$ be a graph and $\cal{V}$ be the set of its vertices. For $X\in \cal{V}$ fixed, we put
\[
e_X(Y):=
\left\{
\begin{array}{cc}
1,& {\rm if} \,Y=X,\\
0, &{\rm otherwise}.
\end{array}
\right.
\]
\item When appropriate we identify \( \Z_+ \) with the set of vertices of a \(1\)-Cayley tree. In particular, \( e_l \), \( l\in\Z_+ \), stands for the function on \( \Z_+ \) defined as above.
\item If $B$ is an operator on the Hilbert space, symbol $\sigma(B)$ will indicate its spectrum.
\item If $\cal A$ is self-adjoint operator defined on $\ell^2(\cal{V})$ and $z\notin\sigma(\cal{A})$,  we will denote the Green's function of
$\cal{A}$ as
\[
G(X,Y,z):= \langle (\cal{A}-z)^{-1}e_Y,e_X\rangle,
\quad X,Y\in \cal{V}\,.
\]
We remark here that the identity
\[
\langle (\cal{A}-z)^{-1}e_Y,e_X\rangle=\langle
e_Y,(\cal{A}^*-\bar{z})^{-1}e_X\rangle=\overline{\langle
(\cal{A}-\bar{z})^{-1}e_X, e_Y\rangle}
\]
implies
\begin{equation}
G(X,Y,z)=\overline{G(Y,X,\bar z)}. \label{ggg1}
\end{equation}
\item If $\mu$ is a finite measure on the real line, then the function
\[
\Theta_\mu(z) := \int_\R\frac{d\mu(x)}{x-z}, \quad z\in \mathbb C,
\]
is called the Stieltjes transform of $\mu$. Clearly, it coincides with the Markov function of $\mu$ up to a sign, i.e., \( \Theta_\mu=-\widehat\mu \), see \eqref{markov}. We introduce this double notation as Markov functions are classical objects in the literature on approximation theory and orthogonal polynomials while Stieltjes transforms are standard in the spectral theory literature.
\end{itemize}


\subsection{Orthogonal polynomials}
\label{S1.2}

Consider a positive measure $\mu$ on $\mathbb{R}$ and assume that $\mu$ satisfies $\supp\,\mu\subseteq [-R,R]$ with some $R>0$. We recall that monic orthogonal polynomials $\{P_n\}, n\in \mathbb{Z}_+$, are defined by the conditions
\begin{equation}\label{1.11}
P_n(x)=x^n+\ldots\;,\quad\int_\R P_n(x)x^l d\mu(x)=0,\quad l\in\{0,\ldots,n-1\}.
\end{equation}
In one-dimensional theory, $\{P_n\}$ are called orthogonal polynomials of the first kind. We  write \eqref{Pnj} as
\[
A_n(z) := P_n^{(1)}(z) = \int_{\mathbb{R}} \frac{P_n(z)-P_n(x)}{z-x}d\mu(x),
\]
which is called the polynomial of the second kind. Notice that \( \deg A_n = n-1 \). Due to orthogonality relations \eqref{1.11}, integral formula for the function of the second kind \eqref{Rnj} can be rewritten as
\[
R_n(z) = \int_\R \left(\frac xz\right)^n\frac{P_n(x)}{z-x}d\mu(x)\,.
\]
Cauchy-Schwarz inequality gives
\begin{equation}\label{sden1_1}
 |R_n(z)|\le 2\|\mu\|^{\frac{1}{2}}R^n|z|^{-n-1}\|P_n\|_{\mu}\,
\end{equation}
 for $|z|>2R$. Polynomials $\{A_n\}$ satisfy the same recurrence as $\{P_n\}$ but with different initial conditions. More precisely, if we let $a_{-1}=-\|\mu\|$, then for $n\in \mathbb{Z}_+$ it holds that
\begin{equation}\label{1.12}
\left\{
\begin{array}{l}
xP_n(x)=P_{n+1}(x)+b_nP_n(x)+a_{n-1}P_{n-1}(x), \, P_{-1}:=0,\, P_0=1, \medskip \\
xA_n(x)=A_{n+1}(x)+b_nA_n(x)+a_{n-1}A_{n-1}(x), \, A_{-1}:=1,\, A_0=0.
\end{array}
\right.
\end{equation}

\subsection{Jacobi matrices}

Let us consider an operator
\begin{equation}\label{1.14}
\cal{H}:=\left[
\begin{array}{ccccc}
b_0 & 1 & 0& 0 &\ldots\\
a_0 & b_1 &1&0&\ldots \\
0 & a_1& b_2& 1&\ldots\\
0 & 0 & a_2 & b_3&\ldots\\
\ldots &\ldots&\ldots&\ldots&\ldots\\
\end{array}
\right]
\end{equation}
that acts on the space of sequences.  Write $\vec{P}:=(P_0,P_1,\ldots)$ and  $\vec{R}:=(R_0,R_1,\ldots)$. It follows from \eqref{1.12} that
\begin{equation}\label{1.15}
\mathcal{H}\vec{P}=x\vec{P}\,,\quad (\mathcal{H}-z)\vec{R}=-{e}_0 \|\mu\|\,,
\end{equation}
thus, formally, $\vec{P}$ is a generalized eigenvector for $\cal{H}$.

Now, we will show how this operator can be symmetrized. To this end, let us introduce
\begin{equation}\label{1.17}
m_n:=\|P_n\|_{\mu}\,.
\end{equation}
Multiplying \eqref{1.12} by $P_{n-1}$ or $P_n$, integrating against the measure $\mu$, and using orthogonality conditions \eqref{1.11} gives
\[
a_{n-1}=\frac{m_n^2}{m_{n-1}^2}>0\,,\quad b_n=\frac{\langle xP_n,P_n\rangle_\mu}{m_n^2}\,,\quad n\in \mathbb{Z}_+.
\]
Notice that \( m_n=(a_{n-1}a_{n-2}\ldots a_1a_0)^{\frac{1}{2}}\|\mu\| \). Denote
\begin{equation}\label{sd_ol}
 \quad p_n:=P_n m_n^{-1},\quad  r_n:=-R_nm^{-1}_n\,.
\end{equation}
Then polynomials  $p_n$ are orthonormal with positive leading coefficients and satisfy
\begin{equation}\label{1.19}
xp_n(x)=c_np_{n+1}(x)+b_np_n(x)+c_{n-1}p_{n-1}(x),\quad c_n:=\sqrt{a_n}\,.
\end{equation}
This equation can be used to easily estimate $\|\{a_n\}\|_\infty$ and $\|\{b_n\}\|_\infty$ in terms of $\supp \,\mu$ only. Indeed, multiplying \eqref{1.19} by $p_{n-1}(x)$ and integrating with respect to $\mu$ gives
\[
c_{n-1}=\int_\mathbb{R} xp_{n-1}(x)p_n(x)d\mu(x)=\int_\mathbb{R}(x-\lambda)p_{n-1}(x)p_n(x)d\mu(x)
\]
with arbitrary $\lambda$. After setting \( \lambda \) to be the midpoint of \( \Delta \) and applying Cauchy-Schwarz inequality, this
yields
\begin{equation}\label{lope}
\|\{c_n\}\|_{\ell^\infty(\mathbb{Z_+})}\le |\Delta|/2,\quad
\,\
\end{equation}
where $\Delta:=   \mathrm{Ch}(\supp\, \mu)$. Next, multiplying \eqref{1.19} by $p_n(x)$ and integrating with respect to $\mu$ gives
\begin{equation}\label{sd_jk}
b_n=\int_{\mathbb{R}} xp_n^2(x)d\mu(x)
\end{equation}
and
 \begin{equation}\label{lope1}
\|\{b_n\}\|_{\ell^\infty(\mathbb{Z_+})}\le \sup_{x\in \Delta}|x|\,.
\end{equation}

The Jacobi matrix $\cal J$, defined by
\begin{equation}\label{1.20}
\mathcal{J}:= \left[
\begin{array}{ccccc}
b_0 & c_0 & 0& 0 &\ldots\\
c_0 & b_1 &c_1&0&\ldots \\
0 & c_1& b_2& c_2&\ldots\\
0 & 0 & c_2 & b_3&\ldots\\
\ldots &\ldots&\ldots&\ldots&\ldots\\
\end{array}
\right],
\end{equation}
is symmetric in $\ell^2(\mathbb{Z}_+)$. Since the sequences $\{a_n\}$ and $\{b_n\}$ are both bounded, the operator $\cal{J}$ is bounded and self-adjoint.
If $\vec{p}:=(p_0,p_1,\ldots),\;\vec{r}:=(r_0,r_1,\ldots)$\,, then \eqref{1.15} and \eqref{sd_ol} yield
\[
\cal{J}\vec{p}=x\vec{p}\,,\quad (\cal{J}-z)\vec{r}=e_0\,.
\]

Similarly to \eqref{sden1_1}, we can write
\[
r_n(z)=-\int_\mathbb{R} \left(\frac xz\right)^n\frac{p_n(x)}{z-x}d\mu(x),\quad
|r_n(z)|<2R^n|z|^{-(n+1)},\, n\in \mathbb{N}, \quad |z|>2R\,,
\]
since \( \mathrm{supp}\mu\subseteq[-R,R] \). Therefore, $\vec{r}\in \ell^2(\mathbb{Z}_+)$ for $|z|>2R$,   and this implies that
\begin{equation}\label{gno1}
\vec{r}=  (\cal{J}-z)^{-1}  e_0,\quad z\notin \sigma(\cal{J}),
\end{equation}
by analyticity in $z$. We will also need the finite sections
\begin{equation}\label{sd_a3}
\cal{J}_N:= \left[
\begin{array}{cccccc}
b_0 & c_0 & 0&                  \ldots &       \ldots &0\\
c_0 & b_1 &c_1&                   \ldots &       \ldots &0 \\
0 & c_1& b_2&      \ldots &      \ldots &  0\\
\ldots & \ldots & \ldots &        \ldots &     \ldots&\ldots\\
0 &0&0&                           \ldots &     c_{N-1} & b_{N}\\
\end{array}
\right]\,
\end{equation}
which are all symmetric matrices.
If $\vec{p}_{N}:=(p_0,\ldots,p_N)$, we get
\begin{equation}
\label{1.21a}
(\cal{J}_N-x)\vec{p}_N=-c_{N}p_{N+1}(x)e_{N}.
\end{equation}

\subsection{Green's functions}

It follows from \eqref{gno1} that
\[
G(e_0,e_0,z)=\langle   (\cal{J}-z)^{-1}e_0,e_0\rangle=-\widehat\mu(z)\|\mu\|^{-1},
\]
which shows that $-\|\mu\|^{-1}\widehat\mu$ is the Stieltjes transform of the spectral measure of ${e}_0$ with respect to $\cal{J}$. Moreover, \eqref{1.21a} implies that
\begin{equation}\label{gf0}
G^{(N)}(e_j,e_N,x)=-\frac{p_j(x)}{c_{N}p_{N+1}(x)}, \quad j\in\{0,\ldots,N\}.
\end{equation}
Hence, the matrix element
\[
M_N(z):= G^{(N)}(e_N,  e_N ,z)=\langle (\cal{J}_N-z)^{-1}e_N,e_N\rangle
\]
is the Stieltjes transform of the spectral measure of ${e}_N$ relative to the operator $\cal{J}_N$ as given by the Spectral Theorem. We also see from \eqref{gf0} that
\begin{equation} \label{r1}
M_N(z)=-\frac{p_N(z)}{c_{N}p_{N+1}(z)}\,.
\end{equation}
Now, take \eqref{1.19} with $n=N$ and divide by $p_N$ to get
\begin{equation}\label{r211}
M_N(z)=\frac{1}{b_N-z-c_{N-1}^2M_{N-1}(z)}\,.
\end{equation}
 Iterating this representation gives a continued
fraction expansion for the rational function $M_N$.

Since $\cal{J}_N$ is self-adjoint,  \eqref{ggg1} yields
\begin{equation}\label{gno2}
G^{(N)}(e_N,e_0,z)=\overline{G^{(N)}(e_0,e_N,\bar
z)}=-\frac{p_0(z)}{c_{N}p_{N+1}(z)}\,,
\end{equation}
because all the coefficients of $p_j$ are real.

Identities \eqref{gno1}, \eqref{r1}, and \eqref{gno2}
establish  remarkable connection between the  spectral
characteristics of $\cal{J}$ and $\cal{J}_N$ and the associated 
orthogonal polynomials $p_n$. In particular, their asymptotics
allows one to write asymptotics of Green's functions. Namely, assume that the measure $\mu$ is supported on $[-1,1]$ and satisfies
the Szeg\H{o} condition
\[
\int_{-1}^1 \frac{\log \mu^\prime(x)}{\sqrt{1-x^2}}dx>-\infty\,.
\]
Then, it is known that (see, e.g., \cite[p. 121, Theorem 5.4]{NikishinSorokin})
\[
p_n(z)=(1+o(1))S(z)\Bigl(z+\sqrt{z^2-1}\Bigr)^n, \quad
n\to\infty, \quad z\in \mathbb{C}\backslash [-1,1]\,,
\]
where $S$, the so-called Szeg\H{o} function, is a function analytic and non-vanishing in $\mathbb{C}\backslash[-1,1]$ which is defined explicitly through $\mu^\prime$. Under these assumptions, we also have
$\lim_{n\to\infty}c_n=1/2,\quad \lim_{n\to\infty}b_n=0$ and
therefore
 $\lim_{N\to\infty}M_N(z)=-2(z-\sqrt{z^2-1}),\,z\in \mathbb{C}\backslash [-1,1]$.\bigskip

\section{Jacobi matrices on trees and MOPs}
\label{sec:JMT}

In this section we assume that $\vec{\mu}$ satisfies \eqref{1.9}.

\subsection{JMs on finite trees and MOPs of the second type}

Fix $\vec{N}=(N_1,\ldots,N_d)\in \mathbb{N}^d$ and a vector ${\vec{\kappa}}\in \mathbb{R}^d$, which satisfies normalization
\begin{equation}
\label{sd_g1}
|{{\vec{\kappa}}}|:={{\kappa}}_1+\cdots+{{\kappa}}_d=1\,.
\end{equation}
We shall define an operator $\cal{K}_{\vec{\kappa},\vec{N}}$, an analog of an \( N\times N \) truncation of the operator $\cal H$ defined in \eqref{1.14}. Domain of $\cal{K}_{\vec{\kappa},\vec{N}}$ consists of functions defined on vertices of a certain finite tree \( \cal T_{\vec N} \) constructed in the following fashion. Truncate $\Z_+^d$ to the rectangle $\mathcal{R}_{\vec{N}}=\{\vec{n}: n_1\leq N_1,\ldots, n_d\leq N_d\}$ and denote by $\mathcal{P}_{\vec{N}}$ the family of all paths of length $|\vec{N}|=N_1+\cdots+N_d$ connecting the points $(0,\ldots,0)$ and $\vec N=(N_1,\ldots,N_d)$ (within a path exactly one of the coordinates is increasing by $1$ at each step). Untwine \( \cal P_{\vec N} \) into a tree \( \cal T_{\vec N} \) in such a way that \( \cal P_{\vec N} \) is in one-to-one correspondence with the paths in \( \cal T_{\vec N} \), where the root of \( \cal T_{\vec N} \), say \( O \), corresponds to \( \vec N \), see Figure~\ref{fig:finite-tree} for $d=2$ and $\vec{N}=(2,1)$.
\begin{figure}[h!]
\includegraphics[scale=.6]{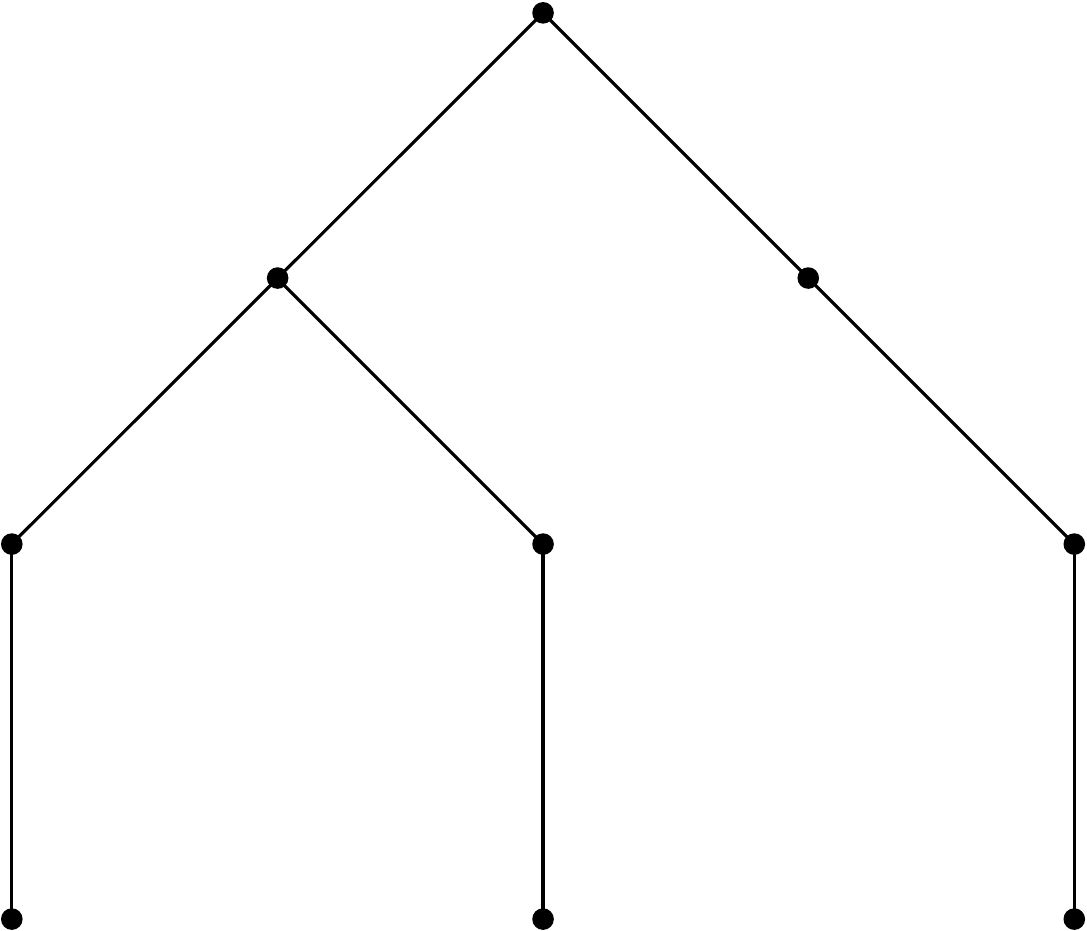}
\begin{picture}(0,0)
\put(-90,157){$(2,1)$}
\put(-137,110){$(1,1)\sim Y_{(p)}$}
\put(-45,110){$(2,0)\sim Z_{(p)}$}
\put(-185,63){$(0,1)\sim Y$}
\put(-95,63){$(1,0)$}
\put(-2,63){$(1,0)\sim Z$}
\put(-185,0){$(0,0)\sim Y_{(ch),2}$}
\put(-95,0){$(0,0)$}
\put(-2,0){$(0,0)\sim Z_{(ch),1}$}
\end{picture}
\caption{Tree for $d=2$ and \( \vec N=(2,1) \).}
\label{fig:finite-tree}
\end{figure}
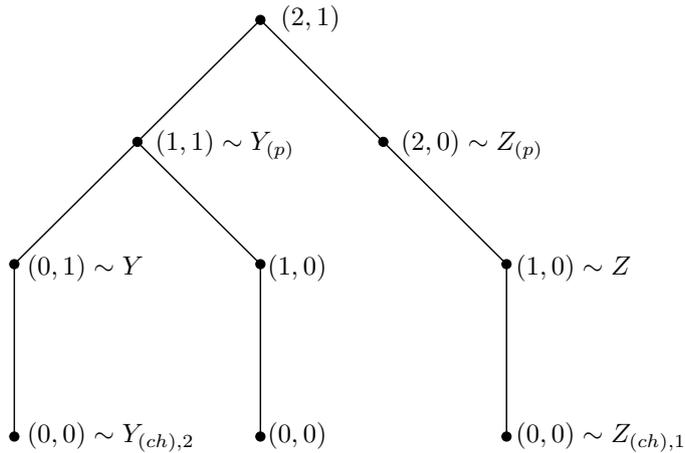
The vertices of \( \cal T_{\vec N} \) correspond to the points of the grid \( \cal R_{\vec N} \) visited along the corresponding path. We denote by $\cal{V}_{\vec{N}}$ the set of these vertices and let  $\Pi$ stand for the projection operator from $\cal{V}_{\vec{N}}$ onto $\cal{R}_{\vec{N}}$.  Given a vertex \( Y\in \cal V_{\vec N} \), we denote by \( Y_{(p)} \) the ''parent'' of \( Y \) and define the following index function on \( \cal V_{\vec N} \):
\[
\ell:\cal V_{\vec N} \to \{1,\ldots,d\}, \quad Y\mapsto \ell_Y \text{ such that } \Pi(Y_{(p)}) = \Pi(Y) + \vec e_{\ell_Y}.
\]
Finally, we denote the ``children'' of \( Y \)  by \( Y_{(ch),l} \), where \( l\in ch(Y):=\{i:n_i>0,\Pi(Y)=(n_1,\ldots,n_d)\} \) and \( \Pi(Y) = \Pi(Y_{(ch),l}) + \vec e_l \) (that is, \( Z=Y_{(ch),l} \) if \( l=\ell_Z\)), see Figure~\ref{fig:finite-tree}.

\begin{remark}
Most of the points in $\cal{R}_{\vec{N}}$  correspond to multiple vertices of the tree \( \cal T_{\vec N} \), so $\Pi^{-1}$, in general, is not uniquely defined.
\end{remark}

\begin{remark}
The number of  children of a vertex \( Y \) is equal to the number of non-zero coordinates of \( \Pi(Y) \). Hence, most of the vertices have exactly \( d \) children.
\end{remark}

To define the operator \( \cal K_{\vec \kappa,\vec N} \), we first define two interaction functions \( V,W:\cal V_{\vec N}\to\R \) with the help of the recurrence coefficients \( \{a_{\n,i},b_{\n,i}\} \) from \eqref{1.5}, \eqref{1.7}. Namely, we set
\[
\left\{
\begin{array}{ll}
V_Y := b_{\Pi(Y),\ell_Y}, & Y\neq O, \medskip \\
V_O := \sum_{j=1}^d \kappa_j b_{\vec{N},j}, & Y=O,
\end{array}
\right. \quad \text{and} \quad
\left\{
\begin{array}{ll}
W_Y := a_{\Pi(Y_{(p)}),\ell_Y}, & Y\neq O, \medskip \\
W_O := 1, & Y=O.
\end{array}
\right.
\]
Then, for any function $f$ defined on $\cal{V}_{\vec{N}}$, the action of the operator $\cal{K}_{\vec{\kappa},\vec{N}}$ can be written in the following form
\[
\left\{
\begin{array}{ll}
(\cal{K}_{\vec{\kappa},\vec{N}} f)_Y := f_{Y_{(p)}} + (Vf)_Y + \sum_{l\in ch(Y)} (Wf)_{Y_{(ch),l}}, & Y\neq O, \medskip \\
(\cal{K}_{\vec{\kappa},\vec{N}} f)_O := (Vf)_O + \sum_{l\in ch(O)} (Wf)_{O_{(ch),l}}, & Y=O.
\end{array}
\right.
\]

\begin{remark}
Clearly, this construction represents untwining $d$ recurrences \eqref{1.7}  at {\it the same} point $\vec{n}\in \mathbb{Z}_+^d$ to equations on the tree - one for each of  {\it many} vertices \( Y \) on $\cal{T}_{\vec{N}}$ that satisfy $\Pi(Y)=\n$.
\end{remark}

\begin{remark}
 The constructed tree $\mathcal{T}_{\vec{N}}$ is not homogeneous since the vertices on the tree representing the points on coordinate planes in $\mathbb{Z}_+^d$ have fewer than $d$ children. However, one can consider the homogeneous infinite rooted tree $\cal{T}, \cal{T}_{\vec{N}}\subset \cal{T}$, with the same root as $\cal{T}_{\vec{N}}$ and extend  $\cal{K}_{\vec{\kappa},\vec{N}}$ to $\cal{T}\backslash \cal{T}_{\vec{N}}$ by setting $\cal{K}_{\vec{\kappa},\vec{N}}=0$. Then, the resulting operator defined on all of $\cal{T}$ decouples into the direct sum of a finite matrix $\cal{K}_{\vec{\kappa},\vec{N}}|_{\cal{T}_{\vec{N}}}$ and the zero operator.
\end{remark}

Let us now consider the polynomials $P_\n(z)$ as a function \( P \) on \( \cal V_{\vec N} \) given by \( P_Y=P_{\Pi(Y)} \), where \( z \) is now treated as a parameter. It follows from \eqref{1.7} and the definition of $\cal{K}_{\vec{\kappa},\vec{N}}$ that $P$ satisfies the following operator equation:
\begin{equation}\label{r3}
\cal{K}_{\vec{\kappa},\vec{N}} P=zP-  \Bigl(\sum_{j=1}^d\kappa_j P_{\vec{N}+\vec{e}_j}(z) \Bigr){e}_O\,.
\end{equation}

\begin{remark}
In the definition of the operator \( \cal K_{\vec\kappa,\vec N} \) the numbers \( \{a_{\n,i},b_{\n,i}\} \) could be absolutely arbitrary. However, \eqref{r3} holds precisely because these numbers come from the recurrence relations \eqref{1.7}.
\end{remark}

Now, we can use  (B) from assumptions \eqref{1.9}  to symmetrize $\cal{K}_{\vec{\kappa},\vec{N}}$ and  produce a self-adjoint operator
$\cal J_{\vec\kappa,\vec N}$, which is an analog of $\cal{J}_N$ defined in \eqref{sd_a3}. To do that, consider a function $m$ defined on
$\cal{V}_{\vec{N}}$ and choose $m$ such that $ \cal J_{\vec\kappa,\vec N} := m^{-1}\cal K_{\vec\kappa,\vec N} m$ is symmetric on
$\ell^2(\cal{V}_{\vec{N}})$. This condition is easy to satisfy by taking $m$ as follows:
\[
m_Y:= \prod_{y\in {\rm path}(Y,O)} \big(W_y\big)^{-\frac 12},
\]
where ${\rm path} (Y,O)$ is the non-self-intersecting path connecting $Y$ and $O$ ($Y$ and $O$ are included in the path). For the resulting self-adjoint operator $\cal J_{\vec\kappa,\vec N}$, which we call Jacobi matrix on a tree, we have
\begin{equation}
\label{r7}
\left\{
\begin{array}{ll}
(\cal{J}_{\vec{\kappa},\vec{N}} f)_Y := \big(W_Y\big)^{1/2}f_{Y_{(p)}} + (Vf)_Y + \sum_{l\in ch(Y)} \big(W_{Y_{(ch),l}}\big)^{1/2}f_{Y_{(ch),l}}, & Y\neq O, \medskip \\
(\cal{J}_{\vec{\kappa},\vec{N}} f)_O := (Vf)_O + \sum_{l\in ch(O)} \big(W_{O_{(ch),l}}\big)^{1/2}f_{O_{(ch),l}}, & Y=O.
\end{array}
\right.
\end{equation}
Furthermore, we get from \eqref{r3}  an identity
\begin{eqnarray}
\label{r4}
\cal J_{\vec\kappa,\vec N} p=zp- \Bigl(\sum_{j=1}^d\kappa_j P_{\vec{N}+\vec{e}_j}(z) \Bigr)e_O, \quad p:=m^{-1} P.
\end{eqnarray}

\begin{remark}
To symmetrize the operator \( \cal K_{\vec\kappa,\vec N} \) we only need the positivity of the numbers \( \{a_{\n,i}\} \), but again to get \eqref{r4} we need the full power of \eqref{1.7}.
\end{remark}

\subsection{JMs on finite trees: Green's functions}

Identity \eqref{r4} gives a formula for the Green's functions of $\cal{J}_{\vec{\kappa},\vec{N}}$:
\begin{equation}\label{fg3}
G^{(\vec{N})}(Y,O,z)=-\frac{p_Y(z)}{\sum_{j=1}^d\kappa_j P_{\vec{N}+\vec{e}_j}(z)   }, \quad z\in \mathbb{C}\backslash \sigma(\cal{J}_{\vec{\kappa},\vec{N}})\,,
\end{equation}
which is an analog of \eqref{gf0}. In fact, if $d=1$ and $\Pi(Y) = n$, we have
\[
p_Y=P_n\sqrt{a_n \cdot \ldots\cdot a_{N-1}}=\sqrt{a_0\ldots a_{N-1}}\frac{P_n}{\sqrt {a_0\ldots a_{n-1}}}=(\sqrt{a_0\ldots a_{N-1}}\|\mu\|)p_n
\]
and $p_Y$ coincides with $p_n$ up to a scalar multiple. Furthermore, by taking $Y=O$ in \eqref{fg3}, we get
\begin{equation}
\label{m11}
\langle (\cal{J}_{\vec{\kappa},\vec{N}}-z)^{-1}e_{O},e_{O}\rangle = -\frac{P_{\vec{N}}(z)}{ \sum_{j=1}^d\kappa_j P_{\vec{N}+\vec{e}_j}(z) }\,.
\end{equation}
If $\vec{\kappa}=\vec e_j$, we get a ratio of two MOPs with neighboring indices similar to \eqref{r1}, that is,
\begin{equation}\label{ratio}
M_{\vec{N}}^{(j)}(z) := \langle (\cal{J}_{\vec e_j,\vec{N}}-z)^{-1}e_{O},e_{O}\rangle = -\frac{P_{\vec{N}}(z)}{P_{\vec{N}+\vec{e}_j}(z)}
\end{equation}
(recall that \( O \) corresponds to the multi-index \( \vec N \) so the analogy with \eqref{r1} is indeed valid). These ratios were already studied in \cite{ApDerVA16} (e.g., formulas (5.5) and (5.6)).  

Divide \eqref{1.7} with $\vec{n}=\vec{N}$ by $P_{\vec{N}}$ to get
\begin{eqnarray}\label{cc1}
x=-\frac{1}{M_{\vec{N}}^{(j)}(x)}+b_{\vec{N},j}-\sum_{l}{a_{\vec{N},l}}{M_{\vec{N}-\vec{e}_l}^{(l)}}(x)\,.
\end{eqnarray}
It is worth mentioning here that $P_{\vec{n}+\vec{e}_j}$ and $P_{\vec{n}+\vec{e}_m}$ are connected by a very simple identity if $j\neq m$. If we subtract recursions \eqref{1.7} from each other and divide the resulting equation by \( P_\n \), we get
\[
\frac{P_{\vec{n}+\vec{e}_j}}{P_{\vec{n}}}
=\frac{P_{\vec{n}+\vec{e}_m}}{P_{\vec{n}}}+b_{\vec{n},m} -b_{\vec{n},j}
\]
implying that
\begin{equation}\label{cc2}
-\frac{1}{M_{\vec{N}}^{(j)}}=-\frac{1}{M_{\vec{N}}^{(m)}}+b_{\vec{N},j}-b_{\vec{N},m}\,\quad \vec{N}\in \mathbb{Z}_+^d\,.
\end{equation}
Iteration of \eqref{cc1}, with application of \eqref{cc2} when the projection of the path hits the margin of $\mathbb{Z}_+^d$, gives a branching continued fraction expansion, which generalizes the standard one obtained from \eqref{r211}.

Finally, consider any $Y$ and apply formula \eqref{ggg1} to \eqref{fg3}. Since $\cal{J}_{\vec{\kappa},\vec{N}}$ is self-adjoint and all the coefficients of $P_{\vec{n}}$ are real, it gives
\[
G^{(\vec{N})}(O,Y,z)=\overline{G^{(\vec{N})}(Y,O,\bar{z})}=-\frac{p_Y(z)}{
\sum_{j=1}^d\kappa_j P_{\vec{N}+\vec{e}_j}(z)
}=-\frac{m^{-1}_YP_Y(z)}{ \sum_{j=1}^d\kappa_j P_{\vec{N}+\vec{e}_j}(z)   }\,.
\]
In particular, taking $Y$ as any point at the bottom of the tree and noticing that $\Pi(Y)=(0,\ldots,0), P_Y=1$, we get
\[
G^{(\vec{N})}(O,Y,z)=-\frac{m^{-1}_{Y}}{ \sum_{j=1}^d\kappa_j P_{\vec{N}+\vec{e}_j}(z) }
\]
which is an analog of \eqref{gno2}.

\subsection{JMs on infinite trees and MOPs of the first type}

Take $\vec{n}\in \mathbb{N}^d$ and consider all paths that connect $(1,\ldots,1)$ with $\vec{n}$. Again, we assume that each path goes from $(1,\ldots,1)$ to $\vec{n}$ by increasing one of the coordinates by $1$ at each step. We consider infinite rooted tree (Cayley tree) with root $O$ that corresponds to $(1,\ldots,1)$. This tree is obtained, as before, by untwining  paths to the lattice, see Figure~\ref{fig:inf-tree} below for $d=2$. We denote this tree by $\cal{T}$ and the set of its vertices by $\cal{V}$. The projection from $\cal{V}$ to $\mathbb{N}^d$ is again denoted by $\Pi$. Every vertex $Y\in \cal{V}, Y\neq O$, has the unique parent, denoted as before by \( Y_{(p)} \), which allows us to define the following index function:
\begin{equation}
\label{tildeell}
\tilde\ell:\cal V \to \{1,\ldots,d\}, \quad Y\mapsto \tilde\ell_Y \text{ such that }  \Pi(Y) = \Pi(Y_{(p)}) + \vec e_{\tilde\ell_Y}.
\end{equation}
With the help of this function we can label the ``children'' of each vertex \( Y \in \cal V \) as $\{Y_{(ch),1},\ldots,Y_{(ch),d}\}$, where we choose index \( l\in\{1,\ldots,d\} \) so that \(  \Pi(Y_{(ch),l}) = \Pi(Y) + \vec e_l \), that is, \( Z=\Pi(Y_{(ch),l}) \) if \( \tilde\ell_Z=l \).

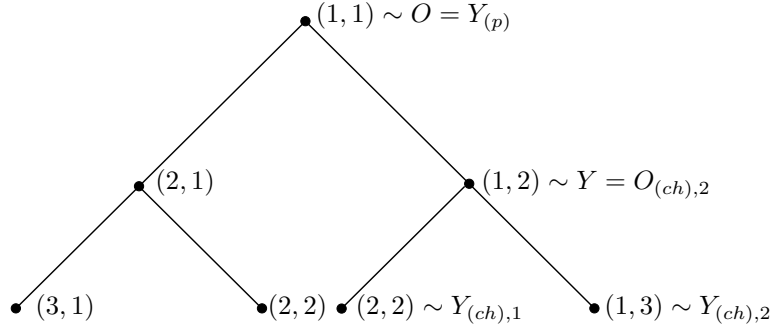
\begin{figure}[h!]
\includegraphics[scale=.6]{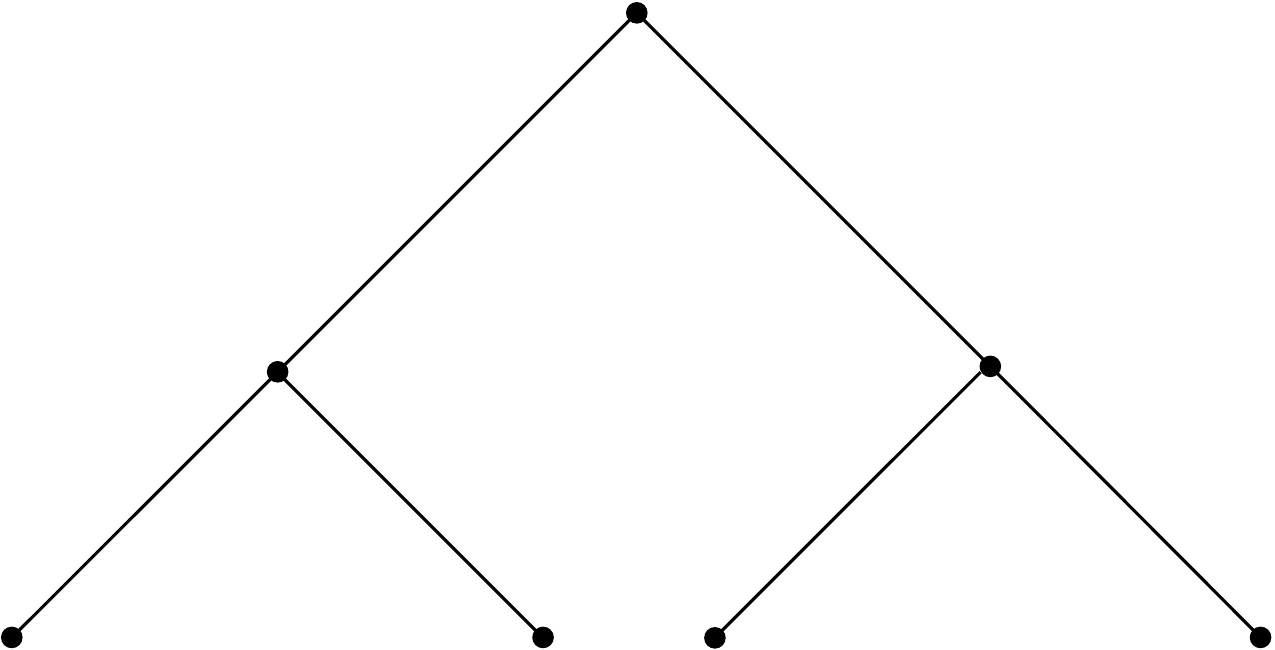}
\begin{picture}(0,0)
\put(-110,110){$(1,1)\sim O=Y_{(p)}$}
\put(-170,47){$(2,1)$}
\put(-48,47){$(1,2)\sim Y=O_{(ch),2}$}
\put(-215,0){$(3,1)$}
\put(-128,0){$(2,2)$}
\put(-95,0){$(2,2)\sim Y_{(ch),1}$}
\put(-2,0){$(1,3)\sim Y_{(ch),2}$}
\end{picture}
\caption{Three generations of the tree \( \mathcal T \) when $d=2$.}
\label{fig:inf-tree}
\end{figure}

Again, let \( \{ a_{\n,i},b_{\n,i} \} \) be the recurrence coefficients from \eqref{1.5}. To define the operator \( \cal R_{\vec \kappa} \) on \( \cal V \), we first define two interaction functions \( \widetilde V,\widetilde W:\cal V\to\R \) by
\[
\left\{
\begin{array}{ll}
\widetilde V_Y := b_{\Pi(Y_{(p)}),\tilde\ell_Y}, & Y\neq O, \medskip \\
\widetilde V_O := \sum_{j=1}^d \kappa_j b_{\vec 1-\vec e_j,j}, & Y=O,
\end{array}
\right. \quad \text{and} \quad
\left\{
\begin{array}{ll}
\widetilde W_Y := a_{\Pi(Y_{(p)}),\ell_Y}, & Y\neq O, \medskip \\
\widetilde W_O := 1, & Y=O,
\end{array}
\right.
\]
where \( \vec 1=( 1,\ldots,1) \) and \( \vec\kappa \) is as in \eqref{sd_g1}. Then, for any function $f\in \ell^2(\cal V)$, the action of the operator $\cal{R}_{\vec{\kappa}}$ can be written in the following form
\begin{equation}
\label{dede}
\left\{
\begin{array}{ll}
(\cal{R}_{\vec{\kappa}} f)_Y := f_{Y_{(p)}} + (\widetilde Vf)_Y + \sum_{l=1}^d (\widetilde Wf)_{Y_{(ch),l}}, & Y\neq O, \medskip \\
(\cal{R}_{\vec{\kappa}} f)_O := (\widetilde Vf)_O + \sum_{l=1}^d (\widetilde Wf)_{O_{(ch),l}}, & Y=O.
\end{array}
\right.
\end{equation}

\begin{remark} Given \( \n\in\N^d \), let \( k\in\{1,\ldots,d\} \) be the number of the coordinates of \( \n \) equal to \(1\). Then for the definition of the operator $\cal{R}_{\vec{\kappa}}$ at \( Y \) with \( \Pi(Y)=\n \) we use one of only \( d-k \) recurrences \eqref{1.5} with excluded ones corresponding to the indices \( j \) such that \( \n-\vec e_j\not\in\N^d \).
\end{remark}

Now, recall formula \eqref{sad1} and  consider  forms $Q_\n(z)$ as a signed-measure-valued function \( Q \) on \( \cal V \) given by \( Q_Y=Q_{\Pi(Y)} \), where \( z \) is treated as a parameter.  Similarly, we can transfer polynomials $A_{\vec{n}}^{(1)}(z), \ldots,A_{\vec{n}}^{(d)}(z)$ to obtain functions $A_{Y}^{(1)}, \ldots,A_{Y}^{(d)}$, respectively, \( Y\in \cal V\), that depend on a parameter \( z \). From our construction and \eqref{1.5}, we have that
\begin{equation}
\label{opRk}
(\cal{R}_{\vec{\kappa}}-x) Q=-   \Bigl(  \sum_{j=1}^{d} \kappa_j Q_{\vec 1-\vec{e}_j}\Bigr)e_{O} = \Bigl(\sum_{i=1}^d \gamma_id\mu_i\Bigr) e_O,
\end{equation}
where the coefficients \( \gamma_i \) can be found explicitly via the relations
\begin{equation}
\label{my-gammai1}
\gamma_i := -\sum_{j=1}^d \kappa_j A_{\vec 1-\vec e_j}^{(i)}
\end{equation}
and the constants \( A_{\vec 1-\vec e_j}^{(i)} \) (these are polynomials of degree at most zero) are such that \( A_{\vec 1-\vec e_j}^{(j)} = 0 \) and
\begin{equation}
\label{my-gammai2}
\left(\begin{matrix} 0 \smallskip \\ \vdots \smallskip \\ 0 \smallskip \\ 1\end{matrix}\right) = \left(\begin{matrix} \int d\mu_1(t) & \cdots & \int d\mu_d(t) \smallskip \\ \vdots & \ddots & \vdots \smallskip \\ \int t^{d-3}d\mu_1(t) & \cdots & \int t^{d-3}d\mu_d(t) \bigskip \\ \int t^{d-2}d\mu_1(t) & \cdots & \int t^{d-2}d\mu_d(t) \end{matrix}\right) \left(\begin{matrix} A_{\vec 1-\vec e_j}^{(1)} \smallskip \\ \vdots \smallskip \\ A_{\vec 1-\vec e_j}^{(d-1)} \smallskip \\ A_{\vec 1-\vec e_j}^{(d)} \end{matrix}\right)
\end{equation}
(when \( d=2 \) the above system retains only the last line; even though the system is written as a matrix with \( d-1 \) rows and \( d \) columns, forcing \( A_{\vec 1-\vec e_j}^{(j)} = 0 \) turns it into a square system). 

Now, define the function \( L \) on \( \cal V \) by setting \( L_Y:=L_{\Pi(Y)} \), see \eqref{Ln} (again, it depends on a parameter \( z \)).  Then, we get from \eqref{dede} that
\[
\int \frac{\langle(\cal{R}_{\vec{\kappa}}-x) Q,e_Y\rangle}{z-x} = \big\langle (\cal{R}_{\vec{\kappa}}-z)L,e_Y\big\rangle + \int Q_{\Pi(Y)}(x) =  \big\langle (\cal{R}_{\vec{\kappa}}-z)L,e_Y\big\rangle ,
\]
where the last equality holds since \( |\Pi(Y)|\geq d\geq 2 \) and therefore \( Q_{\Pi(Y)} \) is always orthogonal to constants by \eqref{1.1}. The above identity, in view of \eqref{opRk} and \eqref{markov}, yields that
\[
(\cal{R}_{\vec{\kappa}}-z) L=\Bigl(\sum_{i=1}^d \gamma_i \widehat\mu_i (z)
\Bigr)e_{O}\,.
\]

Finally, similarly to the case of operators on finite trees, we can symmetrize $\cal{R}_{\vec{\kappa}}$ to get symmetric $\cal{J}_{\vec{\kappa}}$ formally defined via
\begin{equation}
\label{ooo}
\left\{
\begin{array}{ll}
(\cal{J}_{\vec{\kappa}} f)_Y := \big(\widetilde W_Y\big)^{1/2}f_{Y_{(p)}} + (\widetilde Vf)_Y + \sum_{i=1}^d \big(\widetilde W_{Y_{(ch),i}}\big)^{1/2}f_{Y_{(ch),i}}, & Y\neq O, \medskip \\
(\cal{J}_{\vec{\kappa}} f)_O := (\widetilde Vf)_O + \sum_{i=1}^d \big(\widetilde W_{O_{(ch),i}}\big)^{1/2}f_{O_{(ch),i}}, & Y=O.
\end{array}
\right.
\end{equation}
In this case it holds that
\begin{equation}
\label{sd_f33}
(\cal{J}_{\vec{\kappa}}-z)l=\Bigl(\sum_{j=1}^d \gamma_j \widehat\mu_j(z)\Bigr)e_{O}\,,
\end{equation}
where we let
\begin{equation}
\label{muuu}
l_Y:=m_Y^{-1}L_Y, \quad m_Y:= \prod_{y\in {\rm path}(Y,O)} \big(\widetilde W_y\big)^{-1/2},
\end{equation}
and ${\rm path} (Y,O)$ is the non-self-intersecting path connecting $Y$ and $O$. Conditions (C) in assumption \eqref{1.9} imply that  $\widetilde V$ and $\widetilde W$ are bounded. Thus,  $\cal{J}_{\vec{\kappa}}$ is bounded and self-adjoint on $\ell^2(\cal{V})$. 

\begin{remark}
In recent papers, see, e.g., \cite{walt1,VA11}, the recursion parameters $\{a_{\vec{n},j},b_{\vec{n},j}\}$ were computed exactly for some classical weights. In many of these cases,  measures $\{\mu_j\}$ were not compactly supported and at least one of the conditions in (C), \eqref{1.9} was violated. However, our construction can still go through for many of these situations resulting in Jacobi matrix which defines unbounded and formally symmetric operator. We illustrate it with the classical  example of multiple Hermite polynomials defined by absolutely continuous measures given by the Gaussian weights
\[
\mu_j'=e^{-x^2+c_jx}, \quad c_j\ne c_l\,\, {\rm if}\,\, j\ne l,\quad  \quad 1\le j\le  d.
\]
The formula for  Hermite multiple orthogonal polynomials can be written exactly \cite{Soro,Apt_TAMS,w1} and it is known \cite{VA11} that 
\[
b_{\vec{n},j}=c_j/2, \quad a_{\vec{n},j}=n_j/2\,.
\]
Since $a_{\vec{n},j}>0$ for $\vec{n}\in \mathbb{N}^d$, we can repeat our construction to define formally symmetric operators $\cal{J}_{\vec{\kappa}}$ on infinite tree $\cal{T}$. The function $\widetilde{W}_Y$ used in its definition can grow as fast as $\sqrt{|Y|}$ at infinity so $\cal{J}_{\vec{\kappa}}$ is  unbounded in $\ell^2(\cal{V})$. Studying  defect indexes of $\cal{J}_{\vec{\kappa}}$ and existence of self-adjoint extensions in $\ell^2(\cal{V})$ are interesting problems but we choose not to pursue them in this paper. 
\end{remark}

\section{Jacobi matrices on trees and Angelesco systems}
\label{sec:JMAS}

 We continue our discussion for the case when $\vec{\mu}$ forms an Angelesco system (AS).  The foundational result for this section is the following theorem.
\begin{theorem}\label{me1}
If $\vec{\mu}$ forms an Angelesco system, then conditions \eqref{1.9} are satisfied.
\end{theorem}
Its proof is given in Appendix~\ref{appA}.

\subsection{JMs on infinite trees for AS: spectral measures} 

Here we discuss further connections between $\cal{J}_{\vec{\kappa}}$ and MOPs of the first type. Recall that $\Delta_1< \Delta_2<\ldots<\Delta_d$.

\begin{proposition}
\label{prop:4.2}
Let \( l(z) \) be given by \eqref{muuu} and the coefficients \( \gamma_i \) be given by \eqref{my-gammai1}--\eqref{my-gammai2}. Then
\begin{equation}
\label{sd_f4}
l(z)=\Bigl( \sum_{j=1}^d\gamma_j \widehat\mu_j(z) \Bigr)(\cal{J}_{\vec{\kappa}}-z)^{-1}e_{O}
\end{equation}
holds as an identity on the Hilbert space $\ell^2(\cal{V})$ for all $z\notin  (\cup_{j=1}^d \supp\,\mu_j) \cup  \sigma(\cal{J}_{\vec{\kappa}})$. In particular,
\begin{equation}
\label{sd_f5}
G(Y,O,z)= \Bigl(   \sum_{j=1}^d \gamma_j \widehat\mu_j(z)\Bigr)^{-1}  l_Y(z)\,,
\end{equation}
where \( G(Y,O,z) \) is the Green's function for \( \cal J_{\vec \kappa} \).
\end{proposition}
\begin{proof}
Let \( R>0 \) be such that $\supp\, \mu_{j}\subseteq [-R,R]$ for all $j$. It holds that
\[
|L_\n(z)|\le  (|z|-R)^{-|\vec{n}|}, \quad |z|>R,
\]
see \eqref{sd_f2} in Appendix~\ref{appA}. This estimate along with boundedness of $\widetilde W$ implies that $l\in \ell^2(\cal{V})$ provided that $|z|>R_1$, where $R_1$ is sufficiently large. Therefore, we can conclude that \eqref{sd_f33} is satisfied not only formally as a functional identity, but also as an identity on the Hilbert space $\ell^2(\cal{V})$. This, in particular, implies that \eqref{sd_f4} holds for $|z|>R_1$, in which case the functions in both left-hand and right-hand sides are in $\ell^2(\cal{V})$. Now, since for every $Y\in \cal{V}$,  $l_Y(z)$ is analytic away from $\cup_{j=1}^d \supp\,\mu_j$ and $\langle(\cal{J}_{\vec{\kappa}}-z)^{-1}e_{O}, e_{Y}\rangle$ is analytic away from $\sigma(\cal{J}_{\vec{\kappa}})$, they match on the common domain as claimed. Relation \eqref{sd_f5} follows straight from the definition of the Green's function, see the beginning of Section~\ref{sec:CJM}.
\end{proof}

Let the spectral measure of $e_{O}$ with respect to the operator $\cal{J}_{\vec{\kappa}}$ be denoted by $\upsilon_{\vec{\kappa}}$ and recall that Stieltjes transform is defined by
\begin{equation}
\label{sd_f6}
\Theta_{\vec{\kappa}}(z):=G(O,O,z)=:\int_{\mathbb{R}}\frac{d\upsilon_{\vec{\kappa}}(x)}{x-z}\,.
\end{equation}
Formula \eqref{sd_f5}  allows to obtain the following representation for $\Theta_{\vec{\kappa}}$:
\[
\Theta_{\vec{\kappa}}(z)=\frac{L_O(z)}{  \sum_{i=1}^d \gamma_i \widehat\mu_j(z)}=\frac{  \sum_{i=1}^d \widetilde\gamma_i \widehat\mu_i(z)}{  \sum_{i=1}^d \gamma_i \widehat\mu_i(z)}\,,
\]
where the coefficients $\widetilde\gamma_i$ form the solution of the linear system
\begin{equation}
\label{my-gammai3}
\left(\begin{matrix} 0 \smallskip \\ \vdots \smallskip \\ 0 \smallskip \\ 1\end{matrix}\right) = \left(\begin{matrix} \int d\mu_1(t) & \cdots & \int d\mu_d(t) \smallskip \\ \vdots & \ddots & \vdots \smallskip \\ \int t^{d-2}d\mu_1(t) & \cdots & \int t^{d-2}d\mu_d(t) \bigskip \\ \int t^{d-1}d\mu_1(t) & \cdots & \int t^{d-1}d\mu_d(t) \end{matrix}\right) \left(\begin{matrix} \widetilde\gamma_1 \smallskip \\ \vdots \smallskip \\ \widetilde\gamma_{d-1} \smallskip \\ \widetilde\gamma_d \end{matrix}\right).
\end{equation}
In the case $d=2$, the formulas are particularly easy. 

\begin{proposition}
\label{sd_g3}
If $(\mu_1,\mu_2)$ forms an Angelesco system, then
\begin{equation}
\label{spm}
\Theta_{\vec{\kappa}}(z)= \Xi(\mu_1,\mu_2)\frac{\widehat \mu_1(z)\|\mu_2\|-\widehat \mu_2(z)\|\mu_1\|}{\kappa_2
\widehat\mu_1(z) \|\mu_2\|+\kappa_1\widehat\mu_2(z) \|\mu_1\|}, \quad \Xi(\mu_1,\mu_2) := \left(\int_{\mathbb{R}} t\left( \frac{d\mu_2(t)}{\|\mu_2\|}-\frac{
d\mu_1(t)}{\|\mu_1\|}\right)\right)^{-1}.
\end{equation}
\end{proposition}
\begin{proof}
We get from \eqref{my-gammai3} that
\[
 \widetilde\gamma_1 = -\Xi(\mu_1,\mu_2)\|\mu_1\|^{-1} \quad \text{and} \quad \widetilde\gamma_2 = \Xi(\mu_1,\mu_2)\|\mu_2\|^{-1},
\]
and we get from \eqref{my-gammai1}--\eqref{my-gammai2} that
\begin{equation}
\label{my-gammas}
\gamma_1 = - \kappa_2A_{\vec e_1}^{(1)} = -\kappa_2\|\mu_1\|^{-1} \quad \text{and} \quad \gamma_2 = - \kappa_1A_{\vec e_2}^{(2)} = -\kappa_1\|\mu_2\|^{-1},
\end{equation}
which clearly finishes the proof of the proposition.
\end{proof}

This proposition has many applications. For instance, given $\mu_1$ and $\mu_2$, \eqref{spm} allows us to find $\upsilon_{\vec{\kappa}}$. For example, let ${\vec{\kappa}}=(0,1)$. We can take the weak--$(\ast)$ limit $\lim_{\epsilon\to +0}\Im\Theta_{(0,1)}(x+i\epsilon)$, use properties of the Poisson kernel, and write
\begin{equation}
\label{kuk1}
\pi\upsilon_{(0,1)}=\Xi(\mu_1,\mu_2)  \Im^+\left(1-\frac{\|\mu_1\|\widehat\mu_2}{\|\mu_2\|\widehat\mu_1}\right) \\ =
\Xi(\mu_1,\mu_2) \frac{\|\mu_1\|}{\|\mu_2\|}  \left(\chi_{\Delta_1}\widehat \mu_2\Im^+\Bigl(-\frac{1}{\widehat\mu_1}\Bigr)-\chi_{\Delta_2} \widehat\mu_1^{-1}\Im^+\widehat\mu_2 \right),
\end{equation}
where $\Im^+ F$ denotes the weak--$(\ast)$ limit of the imaginary part of functions $F(x+i\epsilon)$ when $\epsilon\to +0$, and \( \chi_E \) is the characteristic functions of a set \(E\). Notice that since $\Delta_1<\Delta_2$ and $\Delta_1,\Delta_2$ do not intersect, $\widehat\mu_2$ is continuous and negative on $\Delta_1$ and  $\widehat\mu_1$ is continuous and positive on $\Delta_2$. Moreover, from the standard properties of convolution with the Poisson kernel, we get \( -\Im^+\widehat\mu_j=\pi\mu_j\). As a corollary, we get 
\[
\supp \, \mu_2\cup \supp\, \Im^+\big(\widehat\mu_1^{-1}\big)\subseteq \sigma(\cal{J}_{(0,1)}).
\]

If the both measures $\mu_1,\mu_2$ are absolutely continuous, given by the weights $w_1,w_2$, respectively, where in addition $w_i>0$ a.e. on $\Delta_i$, \( i\in\{1,2\} \), then we have $\Delta_1\cup\Delta_2\subseteq \sigma(\cal{J}_{(0,1)})$. If we also assume that $w^{-1}_{1}\in L^\infty(\Delta_{1})$, then $\upsilon_{(0,1)}$ is absolutely continuous with respect to Lebesgue measure and
\[
\upsilon_{(0,1)}'=  \Xi(\mu_1,\mu_2)\frac{\|\mu_1\|}{\|\mu_2\|}  \frac{\widehat\mu_1w_2-\widehat\mu_2w_1}{|\widehat\mu_1|^2}\,.
\]
Analogously to the one-dimensional case, the inverse spectral problem can be solved using \eqref{kuk1}.

\begin{proposition}
\label{sd_uni}
Assume that $(\mu_1,\mu_2)$  defines an Angelesco system. If  
\[
\|\mu_1\|, \quad\|\mu_2\|, \quad \Xi(\mu_1,\mu_2), \quad \text{and} \quad \upsilon_{\vec \kappa}
\] 
are known, then $\mu_1,\mu_2$, and $\cal{J}_{\vec \kappa}$ can be found uniquely.
\end{proposition}
\begin{proof}
Set $h:= \widehat\mu_2/\widehat\mu_1$. Then it follows from \eqref{spm} that
\[
h = \frac{\|\mu_2\|}{\|\mu_1\|}\cdot\frac{\Xi(\mu_1,\mu_2)-\kappa_2\Theta_{\vec\kappa}}{\Xi(\mu_1,\mu_2)-\kappa_1\Theta_{\vec\kappa}}.
\]
That is, \( h \) is uniquely defined given \(\|\mu_1\|,\|\mu_2\|, \Xi(\mu_1,\mu_2)\), and \(\upsilon_{\vec \kappa}\). Since $\widehat\mu_1$ is analytic on $\Delta_2$, the problem of finding $\widehat\mu_i$ (and then $\mu_i$), \( i\in\{1,2\} \), can be reduced
to finding $\widehat\mu_2$ from the equation
\[
\widehat\mu_2^+/{\widehat\mu_2^-}={h^+}/{h^-}\,,
\]
where the right-hand side is given a.e. on  $\Delta_2$.  Let \( \widehat\mu_2^+ \) and \( \widehat\mu_2^- \) be the upper and lower non-tangential limits of $\widehat\mu_2$ on the real line, which exist a.e. because $\widehat\mu_2$ is in the Nevanlinna class. Notice also that  $\widehat\mu_2^-=\overline{\widehat\mu_2^+}$  and these functions are different from zero for a.e. $x\in \Delta_2$. If we map $\mathbb{C^+}$ conformally onto $\mathbb{D}$ and consider $\ic\widehat\mu_2$ instead of $\widehat\mu_2$, then the uniqueness of $\widehat\mu_2$ can be deduced from the following claim:

\smallskip

\noindent
{\em If $G$ is analytic in $\mathbb{D}$, $\Re\, G> 0$ in $\mathbb{D}$ and $G/\overline{G}$ is known for a.e. $z\in \mathbb{T}$, then $G$ is defined uniquely up to multiplication by a positive constant.}

\smallskip

Indeed, consider $H:= \log G$ and notice that $H=\log |G|+\ic\arg\, G, \, |\arg \,G|< \pi/2$. Therefore, $H$ belongs to Hardy classes $H^p(\mathbb{D})$ with any $p<\infty$ and so $\log |G|$ and $H$ can be recovered from $\arg\, G$ uniquely up to adding a real constant. On the other hand,  $G/\overline{G}$ defines $\arg\, G$ uniquely  since $|\arg\, G|\le \pi/2$, which finishes justification of the claim.

Therefore, $\widehat\mu_2$ is known up to multiplication by a positive constant. Since $\widehat\mu_2(z)=\|\mu_2\|z^{-1}+O(|z|^{-2})$ when $|z|\to \infty$ and $\|\mu_2\|$ is given, this constant is uniquely defined.
\end{proof}

\begin{remark}
Proposition~\ref{sd_g3}  can be generalized to any $d>2$ with resulting formulas becoming more cumbersome.
\end{remark}

\begin{remark}
From the construction of the operators $\cal{J}_{\vec{\kappa}}$ it is not a priori clear why $\cal{J}_{\vec\kappa}\neq \cal{J}_{\vec\kappa'}$ if $\vec{\kappa}\neq \vec{\kappa'}$.  However, this follows from \eqref{sd_i1} and lemma~\ref{sd_monot} in Appendix~\ref{appA}.
\end{remark}

\subsection{JMs on infinite trees for AS: branching continued fractions}

The branching continued fraction associated with $\cal{J}_{\vec{\kappa}}$ can be constructed in the following way. Choose $Y^\ast\in \cal{V}$ and consider the infinite homogeneous subtree in $\cal{T}$ which has $Y^\ast$ as the root, see Figure~\ref{fig:subtree} for $d=2$ and $Y^*=(1,2)$. We will call it $\cal{T}_{Y^\ast}$ and the set of its vertices is $\cal{V}_{Y^\ast}$. There are $d$ types of these subtrees depending on \( \tilde\ell_{Y^*}\), see \eqref{tildeell}.
\begin{figure}[h!]
\includegraphics[scale=.4]{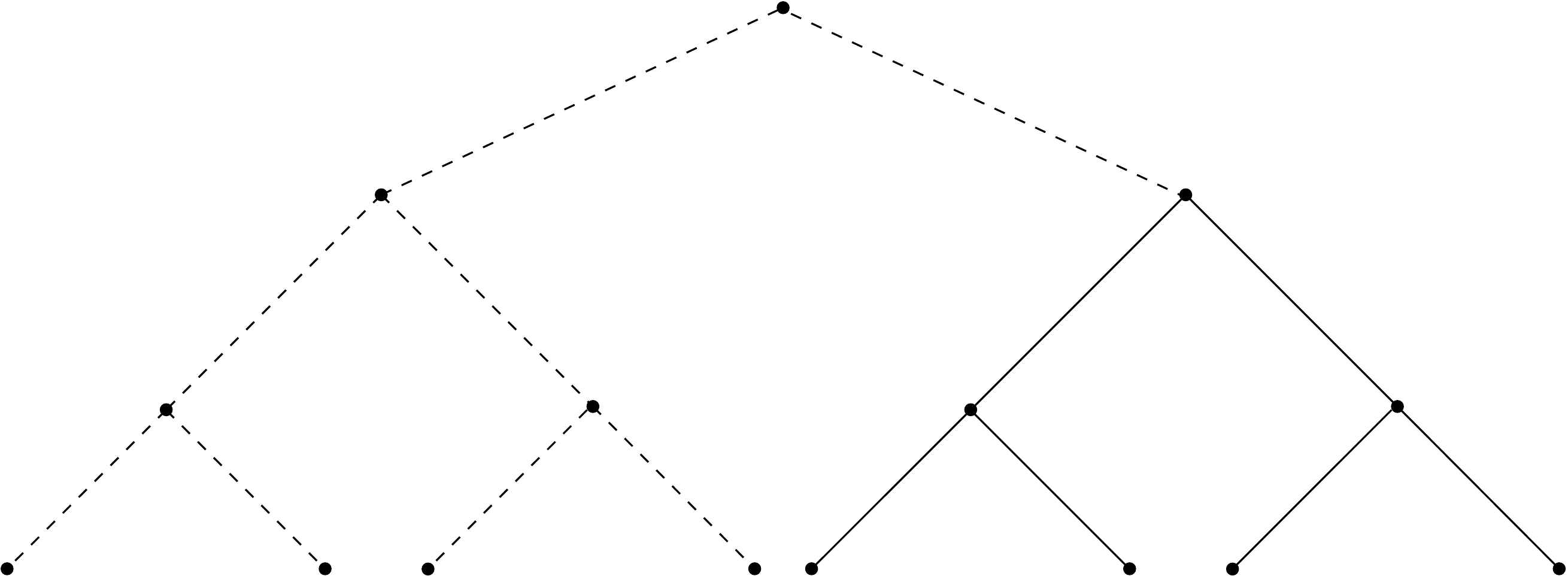}
\begin{picture}(0,0)
\put(-165,115){$(1,1)$}
\put(-75,75){$(1,2)\sim Y^*$}
\put(-256,75){$(2,1)$}
\put(-143,32){$(2,2)$}
\put(-32,32){$(1,3)$}
\put(-188,32){$(2,2)$}
\put(-298,32){$(3,1)$}
\put(-155,-7){$(3,2)$}
\put(-99,-7){$(2,3)$}
\put(-78,-7){$(2,3)$}
\put(-15,-7){$(1,4)$}
\put(-315,-7){$(4,1)$}
\put(-254,-7){$(3,2)$}
\put(-233,-7){$(3,2)$}
\put(-176,-7){$(2,3)$}
\end{picture}
\caption{Three generations of subtree \( \mathcal T_{Y^*} \) with root at \( Y^* \) for $d=2$ (solid lines).}
\label{fig:subtree}
\end{figure}
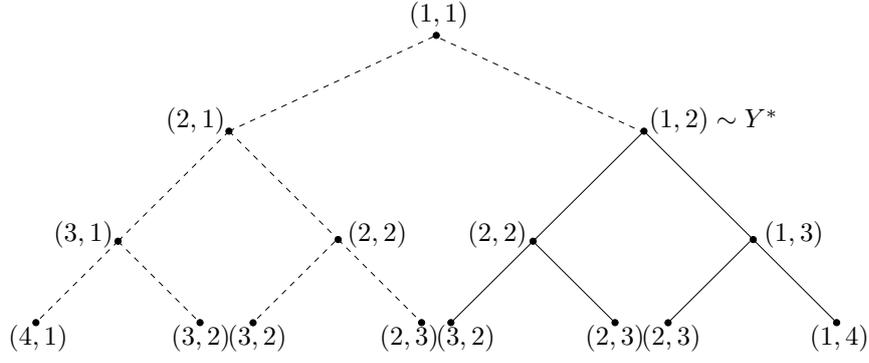
We can define the operator $\cal{R}_{Y^\ast}$ on $\cal{T}_{Y^\ast}$ in the same way as it was done in \eqref{dede} with \( O \) replaced by $Y^\ast$ and ${\vec{\kappa}}=\vec e_{\ell_{Y^*}}$. $\cal{R}_{Y^\ast}$ can be regarded as the restriction of $\cal{R}_{\vec{\kappa}}$ to all $f$ defined on $\cal{V}$ which are zero away from $\cal{V}_{Y^\ast}$. The operator $\cal{R}_{Y^\ast}$ can be symmetrized in the same way as it was done in \eqref{ooo} to produce a self-adjoint operator $\cal{J}_{ Y^\ast}$. By construction, this  $\cal{J}_{ Y^\ast}$ is also equal to restriction of $\cal{J}_{\vec{\kappa}}$ to $\cal{V}_{Y^\ast}$. The Stieltjes transform of the spectral measure of $e_{Y^\ast}$ with respect to the operator $\cal{J}_{Y^\ast}$ defined on $\ell^2(\cal{V}_{Y^\ast})$ is given by
\[
\Theta_{Y^\ast}(z):=\langle (\cal{J}_{Y^\ast}-z)^{-1}e_{Y^\ast},e_{Y^\ast}\rangle\,.
\]
Denote the restriction of $l_Y$ to $\cal{T}_{Y^\ast}$ by $l^{(Y^*)}_Y$. Identity \eqref{sd_f33}, when restricted to $\cal{T}_{Y^\ast}$, implies that
\[
(\cal{J}_{Y^\ast}-z)l^{(Y^\ast)}=-\widetilde W^{\frac 12}_{Y^\ast}    l_{Y^\ast_{(p)}}(z)e_{Y^\ast}\,.
\]
Therefore,
\[
l^{(Y^\ast)}_Y(z)=-\widetilde W^{\frac 12}_{Y^\ast}  l_{Y^\ast_{(p)}}(z)  G_{Y^\ast}(Y,Y^\ast,z),
\]
where $G_{Y^\ast}$ denotes the Green's function of $\cal{J}_{Y^\ast}$. In particular, we get from \eqref{muuu} that
\begin{equation}
\label{sd_f7}
\Theta_{Y^\ast}(z)=-\widetilde{W}^{-\frac 12}_{Y^\ast} \frac{l_{Y^\ast}(z)}{l_{Y^\ast_{(p)}}(z)}=-\frac{L_{\Pi(Y^*)}(z)}{L_{\Pi(Y^*_{(p)})}(z)}\,.
\end{equation}
In the spectral theory of  Schr\"{o}dinger operators on Cayley trees it is known  (see, e.g., \cite{den1}) that the functions $\Theta_{Y^\ast}(z)$ enter into the branching continued fraction for $\Theta_{\vec{\kappa}}$ defined in \eqref{sd_f6}. Let us recall this argument. For $Y^\ast\neq O$, we write equation for $l$ at point $Y^\ast$:
\[
\widetilde{V}_{Y^\ast}l_{Y^\ast} + \widetilde W^{\frac{1}{2}}_{Y^\ast}l_{Y^\ast_{(p)}} + \sum_{j=1}^d \widetilde W^{\frac{1}{2}}_{Y^\ast_{(ch),j}}l_{Y^\ast_{(ch),j}}
=zl_{Y^\ast}\,.
\]
Divide both sides by $l_{Y^\ast}$ and use \eqref{sd_f7} at points $Y^\ast$ and $\big\{Y^\ast_{(ch),i}\big\}$. This gives
\[
\widetilde V_{Y^\ast}-\frac{1}{\Theta_{Y^\ast}(z)}-\sum_{i=1}^d \widetilde W_{Y^\ast_{(ch),i}}\Theta_{Y^\ast_{(ch),i}}(z)
=z\,.
\]
Iterative application of this formula provides the branching continued fraction for $\Theta_{\vec{\kappa}}$. If $d=2$, proposition~\ref{sd_uni} implies that all the entries of this continued fraction can be found uniquely provided that three additional parameters are known.

\subsection{JMs on infinite trees for AS: multiplication operators}

One important aspect of the one-dimensional theory is that the system $\{p_n(x,\mu)\}$ can be used to show that the Jacobi matrix $\cal{J}$ is unitarily equivalent to the multiplication operator defined on $L^2(\mu)$. Indeed, orthogonality conditions give us
\begin{equation}
\label{sd_jj}
e_0(n) = \|\mu\|^{-1/2} \int p_n(x,\mu)d\mu(x), \quad n\in \mathbb{Z}_+,
\end{equation}
and acting on this identity by $\cal{J}^k, k\in \mathbb{Z}_+$, we get
\begin{equation}
\label{sd_ba}
\big(\cal{J}^k e_0\big)(n) = \|\mu\|^{-1/2} \int x^k p_n(x,\mu)d\mu(x), \quad n\in \mathbb{Z}_+\,,
\end{equation}
while $\{p_n\}$ is generalized orthogonal basis of eigenvectors of $\cal{J}$ in the Hilbert space $\ell^2(\mathbb{Z}_+)$. This formula sets the ground for the constructive proof of the Spectral Theorem for $\cal{J}$. In the multidimensional case, some generalizations are possible.
\begin{proposition} If $d=2$, then
\begin{equation}\label{sd_jkl}
e_O(Y) = m^{-1}_Y\int A_Y^{(1)}(x)xd\mu_1(x) + m^{-1}_Y\int A_Y^{(2)}(x)xd\mu_2(x)
\end{equation}
and
\begin{equation}\label{sd_lo1}
\big(\cal{J}^k_{\vec e_1}e_O\big)(Y)=m^{-1}_Y\int T_k(x) A_Y^{(1)}(x)xd\mu_1(x) + m^{-1}_Y\int T_k(x)A_Y^{(2)}(x)xd\mu_2(x)\,,
\end{equation}
where $T_k(x)=x^k+\cdots$ are monic polynomials that can be computed inductively by $T_0(x)=1$ and
\[
T_{k+1}(x)=xT_k(x)+(b_{(0,0),1}-b_{(0,0),2})A_{(1,1)}^{(2)}\int xT_k(x)d\mu_2(x)\,.
\]
Similarly, one can get a formula for $\cal{J}^k_{\vec e_2}e_O$.
\end{proposition}
\begin{proof}
We notice first that \eqref{1.5} implies
\begin{equation}\label{sd_1.5}
xA^{(m)}_{\vec{n}}(x)=A^{(m)}_{\vec{n}-\vec{e}_j}(x)+b_{\vec{n}-\vec{e}_j,j}A^{(m)}_{\vec{n}}(x)+\sum\limits_{i=1}^2 a_{\vec{n},i}
A^{(m)}_{\vec{n}+\vec{e}_i}(x)\,,\quad m,j\in \{1,2\}, \quad \vec{n}\in \mathbb{N}^2\,.
\end{equation}
The formula \eqref{sd_jkl} in proposition follows from the definition of the polynomials of the first type. We notice here that the first integrand, i.e., $m^{-1}_YA_Y^{(1)}$, is a formal eigenfunction of $\cal{J}_{\vec e_1}$ and the second one is a formal eigenfunction of $\cal{J}_{\vec e_2}$ thanks to \eqref{sd_1.5} and  $A_{\vec{e}_1}^{(2)}=A_{\vec{e}_2}^{(1)}=0$. This can serve as a multidimensional analog of identity \eqref{sd_jj} with the striking difference that the formal eigenvectors of two operators are involved. Acting repeatedly on \eqref{sd_jkl} by $\cal{J}_{\vec e_1}$ gives a formula \eqref{sd_lo1} which is similar to \eqref{sd_ba}. Indeed, $T_0(x)=1$. Now, we argue by induction: given \eqref{sd_lo1}, we act on it by $ \cal{J}_{\vec e_1}$ to get
\begin{equation}\label{sd_lo2}
\big(\cal{J}^{k+1}_{\vec e_1}e_O\big)(Y)=m^{-1}_Y\int (xT_k(x)) A_Y^{(1)}(x)xd\mu_1(x) + m^{-1}_Y\int T_k(x)\Bigl(\cal{J}_{\vec e_1}A_Y^{(2)}(x)\Bigr)xd\mu_2(x).
\end{equation}
Next, we notice that \eqref{ooo} and \eqref{sd_i1} yield
\[
(\cal{J}_{\vec e_1}-\cal{J}_{\vec e_2})f=(b_{(0,1),1}-b_{(1,0),2})\langle e_O,f\rangle e_O=(b_{(0,0),1}-b_{(0,0),2})\langle e_O,f\rangle e_O, 
\]
i.e., $\cal{J}_{\vec e_1}$ and $\cal{J}_{\vec e_2}$ are rank-one perturbations of one another and
\[
\cal{J}_{\vec e_1}A^{(2)}(x)=\cal{J}_{\vec e_2}A^{(2)}(x) + (b_{(0,0),1}-b_{(0,0),2})A_O^{(2)} e_O=xA^{(2)}(x)+(b_{(0,0),1}-b_{(0,0),2})A_O^{(2)} e_O\,.
\]
Substituting this into the second term in \eqref{sd_lo2}, we get
\begin{multline*}
\big(\cal{J}^{k+1}_{\vec e_1}e_O\big)(Y)=m^{-1}_Y\int (xT_k(x)) A_Y^{(1)}(x)xd\mu_1(x) + m^{-1}_Y\int (xT_k(x))A_Y^{(2)}(x)xd\mu_2(x) \\ +\left( (b_{(0,0),1}-b_{(0,0),2})A_{(1,1)}^{(2)}\int T_k(x)xd\mu_2(x)  \right)e_O(Y)\,.
\end{multline*}
Now, using \eqref{sd_jkl} we get
\[
T_{k+1}(x)=xT_k(x)+(b_{(0,0),1}-b_{(0,0),2})A_{(1,1)}^{(2)}\int xT_k(x)d\mu_2(x)\,,
\]
which finishes thet proof.
\end{proof}

\begin{remark}
Since all $\{T_k\}$ are linearly independent, the  formula \eqref{sd_lo1} sets the linear isomorphism between  $Span \big\{\cal{J}^k_{\vec e_1}e_O\big\}_{k\in \mathbb{Z}_+}$ and linear space of algebraic polynomials in $x$.
\end{remark}

 The function $f_{Y}:=m^{-1}_{Y}A_{Y}^{(j)}$ formally satisfies an identity $\big(\cal{J}_{\vec e_j}f(x)\big)_Y=xf_Y(x)$ for all $x\in \mathbb{R}$ but, in general, we do not know in what sense it can be regarded as generalized eigenfunction of $\cal{J}_{\vec e_j}$.
 However, if $f(E)\in \ell^2(\cal{V})$ for some $E\in \mathbb{R}$, then $f(E)$ is an actual eigenvector of $\cal{J}_{\vec e_j}$ corresponding to eigenvalue $E$.  Condition $f(E)\in \ell^2(\cal{V})$ can be verified in some cases. For example, take $\vec\kappa=\vec e_2$ and assume that $E$ is an isolated atom in $\mu_2$. Let $\cal{L}$ be the closure of the subspace spanned by vectors $\{e_O,\cal{J}_{\vec e_2}e_O, \cal{J}^2_{\vec e_2}e_O, \ldots \}$. Clearly, $\cal{L}$ is invariant under $\cal{J}_{\vec e_2}$. Let the restriction of $\cal{J}_{\vec e_2}$ to $\cal{L}$ be denoted by
$\widehat{\cal{J}}_{\vec e_2}$. It is a basic fact of the spectral theory of self-adjoint operators, that $\widehat{\cal{J}}_{\vec e_2}$ is unitarily equivalent to a one-dimensional one-sided Jacobi matrix. We do not know if $E$ is an isolated eigenvalue of $\cal{J}_{\vec e_2}$. However,   from \eqref{kuk1} applied to $\cal{J}_{\vec e_2}$, we learn that $E$ is an isolated eigenvalue for $\widehat{\cal{J}}_{\vec e_2}$.
Consider a small contour $\Gamma$ around $E$ which separates it from the rest of the support of $\mu_2$. Since $E$ is an isolated eigenvalue for  $\widehat{\cal{J}}_{\vec e_2}$,  we get representation for  the spectral projection
\[
{\rm Proj}_Ee_O=-\frac{1}{2\pi\ic}\int_\Gamma G(Y,O,z)dz\in \ell^2(\cal{V})\,.
\]
On the other hand, it follows from \eqref{sd_f5} and \eqref{my-gammas} that
\begin{eqnarray*}
-\frac{1}{2\pi\ic}\int_\Gamma G(Y,O,z)dz &=& \frac1{2\pi\ic}\frac{\|\mu_1\|}{m_Y} \int_\Gamma \frac{\int_{\mathbb{R}} \frac{Q_Y(\xi)}{z-\xi}}{\widehat \mu_1(z)}dz = \frac1{2\pi\ic}\frac{\|\mu_1\|}{m_Y} \int_{\mathbb{R}}{\int_{\Gamma}\frac{Q_Y(\xi)dz}{(z-\xi)\widehat \mu_1(z)}} \\
&=& \big(\|\mu_1\|\widehat\mu_1^{-1}(E)\mu_2(\{E\})\big)m^{-1}_YA_Y^{(2)}(E)
\end{eqnarray*}
by residue calculus. Therefore, $m^{-1}_YA_Y^{(2)}\in \ell^2(\cal{V})$ and thus it represents a true eigenvector of~$\cal{J}_{\vec e_2}$.

\subsection{AS with analytic weights: asymptotics of the recurrence coefficients}
\label{ssec:rec}

In this subsection we  describe the asymptotic behavior of the recurrence coefficients \( \{a_{\n,j},b_{\n,j}\} \) from \eqref{1.5}, \eqref{1.7} 
when measures of orthogonality form an Angelesco system \eqref{1.10} with
\[
\supp\,\mu_j=\Delta_j:=[\alpha_j,\beta_j], \quad j\in\{1,\ldots,d\},
\]
 and have analytic non-vanishing densities with respect to the Lebesgue measure on the corresponding interval. The proof of the main theorem is presented in Appendix~\ref{ApB}.

 In what follows, we always assume that
\begin{equation}
\label{multi-indices}
n_i = c_i|\n| + o\big(\n\big), \quad i\in\{1,\ldots,d\}, \quad \vc=(c_1,\ldots,c_d)\in(0,1)^d, \quad |\:\vc\:|:=\sum_{i=1}^dc_i=1.
\end{equation}
When \( d=1 \), i.e., when we have only one interval of orthogonality, it holds that \( \n =n \) and therefore \( \vc = c_1 = 1 \). Even though the middle condition in \eqref{multi-indices} is not satisfied, all the considerations below still apply, however, no results are new in this case.

It is known that the weak asymptotic behavior of multiple orthogonal polynomials is described by the logarithmic potentials of components of a certain vector equilibrium measure \cite{GRakh81}. More precisely, given $\vec c$ as in \eqref{multi-indices}, define
\[
M_{\vec c}\big(\Delta_1,\ldots,\Delta_d\big):=\big\{\vec\nu=(\nu_1,\ldots,\nu_d):~\nu_i\in M_{c_i}(\Delta_i), ~i\in\{1,\ldots,d\}\big\},
\]
where $M_c(\Delta)$ is the collection of all positive Borel measures of mass $c$ supported on $\Delta$. Then it is known that there exists the unique vector of measures $\vec\omega_\vc\in M_\vc\big(\Delta_1,\ldots,\Delta_d\big)$ such that
\[
I[\:\vec\omega_\vc\:] = \min_{\nu\in M_\vc(\Delta_1,\ldots,\Delta_d)} I[\:\vec\nu\:], \qquad I[\:\vec\nu\:] := \sum_{i=1}^d\bigg(2I[\nu_i] + \sum_{k\neq i}I[\nu_i,\nu_k]\bigg),
\]
where $I[\nu_i]:=I[\nu_i,\nu_i]$ and $I[\nu_i,\nu_k]:=-\int\int\log|z-t|\mathrm{d}\nu_i(t)\mathrm{d}\nu_k(z)$. The measure $\omega_{\vc,i}$ might no longer be supported on the whole interval $\Delta_i$ (the so-called \emph{pushing effect}), but in general it holds that
\[
\Delta_{\vc,i} := \supp(\omega_{\vec c,i})  = [\alpha_{\vc,i},\beta_{\vc,i}] \subseteq [\alpha_i,\beta_i], \qquad i\in\{1,\ldots,d\}.
\]

Using intervals \( \Delta_{\vc,i} \) we can define a $(d+1)$-sheeted compact Riemann surface, say $\RS_\vc$, realized in the following way. Take $d+1$ copies of $\overline{\mathbb{C}}$. Cut one of them along the union $\bigcup_{i=1}^d\Delta_{\vc,i}$, which henceforth is denoted by $\RS_\vc^{(0)}$. Each of the remaining copies cut along exactly one interval $\Delta_{\vc,i}$, so that no two copies have the same cut, and denote it by $\RS_\vc^{(i)}$. To form $\RS_\vc$, take $\RS_\vc^{(i)}$ and glue the banks of the cut $\Delta_{\vc,i}$ crosswise to the banks of the corresponding cut on $\RS_\vc^{(0)}$. It can be easily verified that thus constructed Riemann surface has genus \( 0 \). Denote by $\pi_\vc$ the natural projection from $\RS_\vc$ to $\overline{\mathbb{C}}$ (each sheet is simply projected down on to the corresponding copy of the complex plane). We also shall employ the notation $z^{(i)}$ for a point on $\RS_\vc^{(i)}$ with $\pi_\vc(z^{(i)})=z$ and \( \z \) for any point on \( \RS_\vc \) with \( \pi_\vc(\z)=z \).

Since $\RS_\vc$ has genus zero, one can arbitrarily prescribe zero/pole multisets of rational functions on $\RS_\vc$ as long as the multisets have the same cardinality.  Hence, we define \( \chi_\vc(\z) \) to be the rational function on $\RS_\vc$ such that
\begin{equation}
\label{chi}
\chi_\vc\big(z^{(0)} \big) = z + \mathcal O\big(z^{-1}\big) \quad \text{as} \quad z\to\infty.
\end{equation}
This is in fact a conformal map of \( \RS_\vc \) onto the Riemann sphere (it is uniquely defined by \eqref{chi} as all the functions with a single fixed pole are different by an additive constant and therefore prescribing the second term in the Taylor series at \( \infty^{(0)} \) to be zero is equivalent to prescribing a zero). Further, let us define constants  $\big\{A_{\vc,i},B_{\vc,i}\big\}_{i=1}^d$ by
\begin{equation}
\label{ABs}
\chi_\vc\big(z^{(i)} \big) = B_{\vc,i} + A_{\vc,i}z^{-1} + \mathcal O\big(z^{-2}\big) \quad \text{as} \quad z\to\infty.
\end{equation}
Then the following theorem holds.

\begin{theorem}
\label{thm:recurrence}
Assume that the measure $\mu_i$ is absolutely continuous with respect to the Lebesgue measure on \( \Delta_i \) and that the density \( \dd\mu_i(x)/\dd x \)  extends to a holomorphic and non-vanishing function in some neighborhood of $\Delta_i$ for each \( i\in\{1,\ldots,d\} \). Further, let \( \mathcal N_\vc=\{\n\} \) be a sequence of multi-indices for which \eqref{multi-indices} holds. Then the recurrence coefficients $\big\{a_{\n,i},b_{\n,i}\big\}$ from \eqref{1.5}, \eqref{1.7} satisfy
\begin{equation}
\label{limit}
\lim_{\mathcal N_\vc}a_{\n,i} =A_{\vc,i} \quad \text{and} \quad \lim_{\mathcal N_\vc}b_{\n,i} =B_{\vc,i}, \quad i\in\{1,\ldots,d\}.
\end{equation}
\end{theorem}

\begin{remark}
Theorem~\ref{thm:recurrence} as well as all the forthcoming results on asymptotics of MOPs remains valid under more general assumption that \( \dd\mu_i(x)/\dd x \)  is equal to the product of a non-vanishing possibly complex-valued holomorphic function and a so-called \emph{Fisher-Hartwig weight}, see \cite{Y16}. In this case the possibility of normalization \eqref{n_2} and the fact that \( \deg(P_\n) = |\n| \) are no longer immediate, but can be proven to hold for all \( \n\in\mathcal N_\vc \) with \( |\n| \) large enough (in which case the recurrence coefficients $\big\{a_{\n,i},b_{\n,i}\big\}$ are well defined). However, we opted not to pursue this generalization as it is technical and not conceptual in nature.
\end{remark}

\begin{remark}
When \( d =1 \) and we denote the single interval of orthogonality by \( [\alpha,\beta] \), the corresponding conformal map \( \chi \) can be explicitly written as
\[
\chi\big(z^{(k)}\big) = \frac{z-(\alpha+\beta)/2 - (-1)^k\sqrt{(z-\alpha)(z-\beta)}}2,
\]
for \( k\in\{0,1\} \), and therefore \( A= (\beta-\alpha)^2/16 \), \( B = (\beta+\alpha)/2 \), as expected.
\end{remark}

Since \( \chi_\vc(\z) \) is a conformal map, all the numbers \( B_{\vc,i} \) are distinct. Hence, the following corollary is an immediate consequence of theorem~\ref{thm:recurrence} and \cite[Theorem~1.1]{MR3489551}.

\begin{corollary}
\label{cor:typeIIa}
Under the conditions of theorem~\ref{thm:recurrence}, let polynomials \( P_\n(x) \) satisfy \eqref{1.2}. Then it holds that
\[
\lim_{\mathcal N_\vc} P_\n(z)/P_{\n+\vec e_j}(z) = \left(\chi_\vc\big(z^{(0)} \big) - B_{\vc,j} \right)^{-1}
\]
uniformly on closed subsets of \( \overline{\mathbb{C}}\setminus \bigcup_{i=1}^d \Delta_{\vc,i} \) for every \( j\in\{1,\ldots,d\} \) .
\end{corollary}

\subsection{JMs on finite trees for AS: convergence}

Our main goal in this subsection is to illustrate how connection between theory of MOPs and Jacobi matrices can be used to obtain results about MOPs. For that purpose, we will focus on ratio asymptotics. 

\begin{proposition}
Let \( \vc\in(0,1)^d \) and \( \mathcal N_\vc = \big\{\vec N\big\} \) be as in \eqref{multi-indices} (replace \( \n \) with \( \vec N \)). Suppose that $\vec\mu$ forms an Angelesco system for which the recurrence coefficients satisfy
\begin{equation}
\label{sd_jo7}
\lim_{\mathcal N_\vc}a_{\n,i} =A_{\vc,i} \quad \text{and} \quad \lim_{\mathcal N_\vc}b_{\n,i} =B_{\vc,i}, \quad i\in\{1,\ldots,d\}\,.
\end{equation}
Then the following limits exist:
\[
M_{\vc}^{(j)}(z) := -\lim_{\mathcal N_\vc}   \frac{P_{\vec{N}}(z)}{P_{\vec{N}+\vec{e}_j}(z)}, \quad j\in\{1,\ldots, d\},
\]
and the convergence is uniform on closed subsets of \( \overline{\mathbb{C}}\setminus \bigcup_{i=1}^d \Delta_i \).
\end{proposition}  

This result  slightly generalizes part of \cite[Theorem~1.1]{MR3489551}. It can be used to give alternative proof to corollary \ref{cor:typeIIa}.

\begin{proof}
Consider operators $\{\cal{J}_{\vec{\kappa}, \vec{N}}\}$ introduced in \eqref{r7} for \( \vec N\in\mathcal N_\vc \). Thanks to a remark given right before the formula \eqref{r3}, we can assume that all these operators are defined on the single infinite tree $\cal{T}$. From the results of Appendix~\ref{appA}, we also know that coefficients in these operators are uniformly bounded, which implies $ \sup_{\vec{N}\in\mathcal N_\vc} \| \cal{J}_{\vec{\kappa}, \vec{N}}   \|<\infty$.

On the infinite $d+1$ homogeneous tree $\cal{T}$ with root at $O$, define operator $ \cal{J}_{\vec{\kappa},\vec{c}} $ obtained by formally taking the limit in \eqref{r7} and using \eqref{sd_jo7}:
\begin{equation}
\label{Jkc}
\left\{
\begin{array}{ll}
(\cal{J}_{\vec{\kappa},\vec{c}} f)_Y := B_{\vc,i}   f_Y+(A_{\vc,i} )^{\frac{1}{2}}f_{Y_{(p)}}+\sum_{j=1}^d( A_{\vc,j}    )^{\frac{1}{2}}f_{Y_{(ch),j}}, & Y\neq O, \medskip \\
(  \cal{J}_{\vec{\kappa},\vec{c}}   f)_O := \sum_{i=1}^dB_{\vc,i}\kappa_i   f_O+\sum_{j=1}^d( A_{\vc,j}    )^{\frac{1}{2}}f_{O_{(ch),j}}, & Y=O,
\end{array}
\right.
\end{equation}
where $Y_{(p)}$ has $d$ children each corresponding to the index $i\in \{1,\ldots,d\}$.

First, we claim that $ \cal{J}_{\vec{\kappa},\vec{N}}   \to \cal{J}_{\vec{\kappa},\vec{c}}$ in the strong operator sense, i.e.,
\[
\big\|\big(  \cal{J}_{\vec{\kappa},\vec{N}}   - \cal{J}_{\vec{\kappa},\vec{c}} \big)f\big\|_{\ell^2(\cal{V})} \to 0
\]
for every fixed $f\in \ell^2(\cal{V})$. Indeed,  let $\chi_{|X|<\rho}$ be the characteristic function of the ball in $\cal{T}$ with center at $O$ and radius $\rho$. Given any $\epsilon>0$, there is $\rho_\epsilon$ such that
\[
\|f-f\chi_{|X|<\rho_\epsilon}\|_{\ell^2(\cal{V})} \le \epsilon\,.
\]
Since coefficients $\{a_{\vec{n},j}\}$ and $\{b_{\vec{n},j}\}$ are uniformly bounded, we have
\[
\big\|\big(    \cal{J}_{\vec{\kappa},\vec{N}}   - \cal{J}_{\vec{\kappa},\vec{c}}  \big)(f\cdot\chi_{|X|\ge \rho_\epsilon})\big\|_{\ell^2(\cal{V})} \le C\epsilon
\]
uniformly in $\vec{N}$. Having $\epsilon$ and $\rho_\epsilon$ fixed, we get
\[
\big\|\big(   \cal{J}_{\vec{\kappa},\vec{N}}   - \cal{J}_{\vec{\kappa},\vec{c}}   \big)(f\cdot\chi_{|X|< \rho_\epsilon})\big\|_{\ell^2(\cal{V})} \to 0
\]
by our assumptions \eqref{sd_jo7}. This proves our claim.

Next, the Second Resolvent Identity from perturbation theory of operators gives
\[
(\cal{J}_{\vec e_j,\vec{N}}-z)^{-1}e_O=(\cal{J}_{\vec e_j,\vec{c}}-z)^{-1}e_O-(\cal{J}_{\vec e_j,\vec{N}}-z)^{-1}(\cal{J}_{\vec e_j,\vec{N}}-\cal{J}_{\vec e_j,\vec{c}})(\cal{J}_{\vec e_j,\vec{c}}-z)^{-1}e_O
\]
for $z\in \mathbb{C}\backslash \mathbb{R}$.  Since $\|(\cal{J}_{\vec e_j,\vec{c}}-z)^{-1}\|\le |\Im z|^{-1}$ by the Spectral Theorem,  we can take $|\vec{N}|\to\infty$ and use the above claim to obtain
\begin{equation}
\label{sd_ll9}
\lim_{|\vec{N}|\to\infty,\vec N\in\cal N_\vc}(\cal{J}_{\vec e_j,\vec{N}}-z)^{-1}e_O=(\cal{J}_{\vec e_j,\vec{c}}-z)^{-1}e_O
\end{equation}
and this convergence is uniform in $z$ over compacts in $\mathbb{C}^+$ and $\mathbb{C}^-$. Now, recall the notations \eqref{m11} and \eqref{ratio} for the resolvent matrix element
$$
M_{\vec{N}}^{(j)}(z) := \big\langle(\cal{J}_{\vec e_j,\vec{N}}-z)^{-1}e_{O},e_{O}\big\rangle\,=\, -\frac{P_{\vec{N}}(z)}{P_{\vec{N}+\vec{e}_j}(z)}\,.
$$
Thus, from \eqref{sd_ll9}, we get the required ratio asymptotics and $M_{\vec{c}}^{(j)}=\langle (\cal{J}_{\vec e_j,\vec{c}}-z)^{-1}e_O,e_O\rangle$. To extend this convergence to closed subsets of \( \overline{\mathbb{C}}\setminus \bigcup_{i=1}^d \Delta_i \), we only need to notice that interlacing property of zeros implies that functions $M_{\vec{N}}^{(j)}$ are uniformly bounded and analytic on them. Thus, by normal family argument we can prove that $M_{\vec{c}}^{(j)}$ are analytic there and the uniform convergence extends to these closed sets as well.
\end{proof}

Two remarks are in order now. 

\begin{remark}
Under the conditions of theorem~\ref{thm:recurrence}, we use corollary \ref{cor:typeIIa} to get
\begin{equation}\label{AlgF}
-M_{\vec{c}}^{(j)}(z) \,=\, A_{\vc,j}^{-1}\Upsilon_{\vc,j}\big(z^{(0)}\big), \quad z\in  \mathbb{\overline{C}}\setminus \bigcup_{i=1}^d \Delta_{\vec{c},i},
 \end{equation}
where\footnote{compare with formula \eqref{sd_jk1} in Appendix~\ref{ApB}} \( \Upsilon_{\vc,i} := A_{\vc,i}/\big(\chi_\vc-B_{\vc,i}\big) \), \( i\in\{1,\ldots,d\} \). Taking the limit in formulas \eqref{ratio} and \eqref{cc1}, we obtain
\begin{eqnarray}\label{cccc1_1}
z=\frac{A_{\vec{c},j}}{\Upsilon_{\vc,j}\big(z^{(0)}\big)   }+B_{\vec{c},j}+\sum_{i=1}^d  \Upsilon_{\vc,i}\big(z^{(0)}\big), \quad j\in\{1,\ldots,d\}.
\end{eqnarray}
In other words, functions $\Upsilon_{\vc,j}\big(z^{(0)}\big)$ define a solution to a system of $d$ algebraic equations and each of them, when multiplied by \(-1\), is in Nevanlinna class in $\mathbb{C}^+$.
\end{remark}

\begin{remark}
We can repeat the argument given right after formula \eqref{sd_f7}  to show that the matrix element of the resolvent operator $M_{\vec{c}}^{(j)}=\langle (\cal{J}_{\vec e_j,\vc}-z)^{-1}e_{O},e_{O}\rangle$ satisfies  equation similar to \eqref{cccc1_1}. Fix \( j\in\{1,\ldots,d\} \). Denoting the Green's function
$$
u:=(\cal{J}_{\vec e_j,\vec{c}}-z)^{-1}e_{O},
$$
we have $(\cal{J}_{\vec e_j,\vec{c}}-z)u=e_{O}$, which can be rewritten using \eqref{Jkc} as
\begin{equation}
\label{my-aux1}
B_{\vec{c},j}{u}_O+\sum_{l=1}^d (A_{\vc,l})^{1/2}{u}_{O_{(ch),l}}  \,=\, z {u}_{O}+1\,.
\end{equation}
As \( M^{(j)}_{\vec{c}} = \langle u,e_O\rangle = u_O\), that is equivalent to
\begin{equation}
\label{star1}
B_{\vec{c},j}M^{(j)}_{\vec{c}}+\sum_{i=1}^d ( A_{\vc,i})^{1/2}{u}_{O_{(ch),i}}  \,=\, z M^{(j)}_{\vec{c}}+1\,.
\end{equation}
Let us write $O_i:=O_{(ch),i}$, \( i\in\{1,\ldots,d\} \). Then we get from \eqref{Jkc} that
$$
z\, u_{O_j}\,=\,(A_{\vc,j} )^{\frac{1}{2}} u_O+ B_{\vc,j} u_{O_j}\,+\,\sum_{i=1}^d
( A_{\vc,i}    )^{\frac{1}{2}}u_{(O_j)_{(ch),i}}\,,
$$
or equivalently
\begin{equation}
\label{my-aux2}
B_{\vc,j}\frac{ -u_{O_j}}{(A_{\vc,j} )^{\frac{1}{2}} u_O}\,+\,\sum_{i=1}^d ( A_{\vc,i}    )^{\frac{1}{2}}\frac{-u_{(O_j)_{(ch),i}}}{(A_{\vc,j} )^{\frac{1}{2}} u_O} = 1 + z\, \frac{-u_{O_j}}{(A_{\vc,j} )^{\frac{1}{2}} u_O}.
\end{equation}
Let us denote by $(\cal{J}_{\vec e_j,\vc})_i$ the truncation of the operator $\cal{J}_{\vec e_j,\vc}$ to the subtree $\cal{T}^{(i)}$ with root at $O_i$. Further, let \( u^{(i)} \) be the Green's function for \( (\cal{J}_{\vec e_j,\vc})_i \), \( i\in\{1,\ldots,d\} \). By comparing \eqref{my-aux1} and \eqref{my-aux2}, we immediately see that 
$$
u^{(j)}_{O_j} = -\, \frac{u_{O_j}}{(A_{\vc,j} )^{\frac{1}{2}} u_O}.
$$
Identifying \( \ell^2\big(\cal V^{(i)}\big) \) with \( \ell^2(\cal V) \) in a standard way, we see that the operators \( (\cal{J}_{\vec e_j,\vc})_j \) and \( \cal{J}_{\vec e_j,\vc} \) are identical and therefore \( u^{(j)}_{O_j}=u_O=M_\vc^{(j)} \). Hence, it holds that
$$
u_{O_j}\,=\,-\,(A_{\vc,j} )^{\frac{1}{2}}\,\big( M_{\vec{c}}^{(j)} \big)^2.
$$
Substituting this result into \eqref{star1} we obtain
\begin{equation}\label{star2}
B_{\vec{c},j}M^{(j)}_{\vec{c}}-A_{\vc,j} \,\big( M_{\vec{c}}^{(j)} \big)^2 + \sum_{i=1, i\neq j}^d ( A_{\vc,i})^{1/2}{u}_{O_{(ch),i}} \,=\, z M^{(j)}_{\vec{c}}+1\,.
\end{equation}
An analogous argument on subtree $\cal{T}^{(i)}$,  $i \neq j$, yields that
$$
u_{O_i}\,=\,-\,(A_{\vc,i} )^{\frac{1}{2}}\, M_{\vec{c}}^{(i)}M_{\vec{c}}^{(j)}.
$$
Substituting this result into \eqref{star2}, we arrive at
\begin{equation*}
z=-\frac{1}{M_{\vec{c}}^{(j)}}+B_{\vec{c},j}-\sum_{i=1}^d{A_{\vec{c},i}}{M_{\vec{c}}^{(i)}},
\end{equation*}
which is consistent with \eqref{cccc1_1}.
\end{remark}

\subsection{AS with analytic weights: asymptotics of MOPs}

This subsection is the continuation of Section~\ref{ssec:rec}. In what follows, we shall set \( F^{(k)}(z) := F(z^{(k)}) \) for a function \( F \) on a given Riemann surface. It will also be convenient for us to set
\begin{equation}
\label{weights}
\frac{\dd\mu_i(x)}{\dd x} = -\frac{\rho_i(x)}{2\pi\ic},
\end{equation}
where,  as before, we assume that \( \rho_i(x) \) extends to a holomorphic and non-vanishing function in some neighborhood of \( \Delta_i \). Put
\[
w_{\vc,i}(z):=\sqrt{(z-\alpha_{\vc,i})(z-\beta_{\vc,i})}
\]
to be the branch holomorphic outside of $\Delta_{\vc,i}$ normalized so that $w_{\vc,i}(z)/z\to1$ as $z\to\infty$. Observe that
\[
(\rho_iw_{\vc,i+})(x) = 2\pi|w_{\vc,i}(x)|(d\mu_i(x)/d x)>0, \quad x\in \Delta_{\vc,i}^\circ:=(\alpha_{\vc,i},\beta_{\vc,i}),
\]
where \( w_{\vc,i+}(x) \) stands for the non-tangential limit of \( w_{\vc,i}(z) \) on \( \Delta_{\vc,i} \) taken from the upper half-plane. The following facts have been established in  \cite[Proposition~2.4]{Y16}.

\begin{proposition}
 There exists the unique up to a multiplication by a \( (d+1) \)-st root of unity set of functions $S_\vc^{(k)}(z)$, \( k\in\{0,\ldots,d\} \), such that
\begin{itemize}
\item \( S_\vc^{(0)}(z) \) is non-vanishing and holomorphic in $\overline{\mathbb{C}}\setminus\bigcup_{i=1}^d \Delta_{\vc,i}$ and \( S_\vc^{(i)}(z) \) is non-vanishing and holomorphic in $\overline{\mathbb{C}}\setminus \Delta_{\vc,i}$, \( i\in\{1,\ldots,d\} \);
\item $S_\vc^{(0)},S_\vc^{(i)}$ have continuous traces on \( \Delta_{\vc,i}^\circ \) that satisfy \( S_{\vc\:\pm}^{(i)}(x) = S_{\vc\:\mp}^{(0)}(x) \big(\rho_iw_{\vc,i+}\big)(x) \) there;
\item it holds that \( |S_\vc^{(0)}(z)| \sim |S_\vc^{(i)}(z)|^{-1} \sim |z-z_0|^{-1/4} \) as \( z\to z_0\in\{\alpha_{\vc,i},\beta_{\vc,i}\} \), $i\in\{1,\ldots,d\}$\footnote{\( A(z)\sim B(z) \) as \( z\to z_0 \) means that the ratio \( A(z)/B(z) \) is uniformly bounded away from zero and infinity as \( z\to z_0\).} and \( \prod_{k=0}^dS_\vc^{(k)}(z)\equiv1 \), \( z\in\mathbb{C} \).
\end{itemize}
These functions are continuous with respect to the parameter \( \vc \), i.e., \( S_\vc^{(i)}(z) \to S_{\vc_0}^{(i)}(z) \) for each \( z\in\overline{\mathbb C}\setminus\{\alpha_{\vc_0,i},\beta_{\vc_0,i}\} \) (including the traces on \( \Delta_{\vc_0}^\circ \)) as \( \vc\to\vc_0\in(0,1)^d \) for each \( i\in\{1,\ldots,d\} \).
\end{proposition}

\begin{remark}
In the single interval \( [\alpha,\beta] \), i.e., in the case \( d=1 \), let
\[
S_\rho(z) := \exp\left\{\frac{w(z)}{2\pi\ic}\int_\alpha^\beta \frac{\log(\rho w_+)(x)}{z-x}\frac{dx}{w_+(x)}\right\},
\]
where \( w(z):=\sqrt{(z-\alpha)(z-\beta)} \), be the classical Szeg\H{o} function for a log-integrable positive weight \( \rho \). Then it is easy to check that \( S^{(0)} = S_\rho \) and \( S^{(1)} = 1/S_\rho \).
\end{remark}

Hereafter, we use for simplicity subindex \( \n \) instead of the normalized subindex \( \n/|\n|\in(0,1)^d \). For example, we shall write \( \RS_\n \) instead of \( \RS_{\n/|\n|} \). Let \( \Phi_\n(\z) \) be a rational function on \( \RS_\n \) with zero/pole divisor given by
\[
n_1\infty^{(1)} + \cdots + n_p \infty^{(p)} - |\n|\infty^{(0)}
\]
normalized so that \( \prod_{k=0}^d\Phi_n^{(k)}(z)\equiv 1 \) (such a normalization is possible since this product is necessarily a bounded entire function and therefore is a constant, and it is unique up to a multiplication by a \( (d+1) \)-st root of unity). It can be shown \cite[Proposition~2.1]{Y16} that
\[
\frac1{|\n|}\log\big|\Phi_\n(\z)\big| =
\left\{
\begin{array}{ll}
-V^{\omega_\n}(z) + \frac1{d+1}\sum_{k=1}^d\ell_{\n,k}, & \z\in\RS_\n^{(0)}, \medskip \\
V^{\omega_{\n,i}}(z) - \ell_{\n,i} + \frac1{d+1}\sum_{k=1}^d\ell_{\n,k}, & \z\in\RS_\n^{(i)}, \quad i\in\{1,\ldots,d\},
\end{array}
\right.
\]
for certain constants \( \ell_{\n,i} \), \( i\in\{1,\ldots,d\} \), where \( \omega_\n:=\sum_{i=1}^d\omega_{\n,i} \) and \( V^\omega(z):=-\int\log|z-t|\dd\omega(t) \) is the logarithmic potential of \( \omega \). It is of course true that \( \omega_{\n,i} \cws \omega_{\vc,i}\) and \( \ell_{\n,i}\to\ell_{\vc,i} \) as \( |\n|\to\infty \), \( \n\in\mathcal N_\vc \), for each \( i\in\{1,\ldots,d\} \), where \( \cws \) denotes weak-$(\ast)$ convergence of measures.

\begin{theorem}
\label{thm:typeII}
Let the measures \( \mu_i \), \( i\in\{1,\ldots,d\} \), be as in \eqref{weights} and polynomials \( P_\n(x) \) satisfy \eqref{1.2}. Further, let \( \mathcal N_\vc=\{\n\} \) be a sequence for which \eqref{multi-indices} holds. Then it holds for \( \n\in\mathcal N_\vc \) that
\[
\left\{
\begin{array}{lll}
P_\n(z) &=& \left(1 + \mathcal O\big(|\n|^{-1}\big)\right) \gamma_\n\big(S_\n\Phi_\n\big)^{(0)}(z), \medskip \\
P_\n(x) &=& \left(1 + \mathcal O\big(|\n|^{-1}\big)\right) \gamma_\n\big(S_\n\Phi_\n\big)_+^{(0)}(x) + \left(1 + \mathcal O\big(|\n|^{-1}\big)\right) \gamma_\n\big(S_\n\Phi_\n\big)_-^{(0)}(x),
\end{array}
\right.
\]
where the first relation holds uniformly on closed subsets of \( \overline{\mathbb{C}} \setminus \bigcup_{i=1}^d \Delta_{\n,i} \), the second one holds uniformly on compact subsets \( \bigcup_{i=1}^d \Delta_{\n,i}^\circ \), and \( \gamma_\n \) is a constant such that \(\displaystyle \lim_{z\to\infty}\gamma_\n z^{|\n|}\big(S_\n\Phi_\n\big)^{(0)}(z)=1 \).
\end{theorem}

\begin{remark}
Observe that in the statement of the above theorem the functions \( S_\n^{(k)} \) can be replaced by their limits \( S_\vc^{(k)} \) at the expense of possibly loosing \( |\n|^{-1}\)-rate of convergence. Moreover, if the sequence \( \mathcal N_\vc \) is such that no pushing effect occurs for all its indices large enough, then \( \RS_\n = \RS_\vc \) and \( S_\n^{(k)} = S_\vc^{(k)} \) any way for all such indices.
\end{remark}

Let \( \Pi_\n(\z) \) be a rational function on \( \RS_\n \) with the zero/pole divisor and the normalization given by
\[
2\big(\infty^{(1)} + \cdots + \infty^{(d)} \big) - \mathcal D_\n \quad \text{and} \quad \Pi_\n^{(0)}(\infty) =1,
\]
where \( \mathcal D_\n \) is the divisor of ramification points of \( \RS_\n \). When \( d=1 \), it in fact holds that
\[
\Pi\big(z^{(k)}\big) = \frac{z-(\alpha+\beta)/2 + (-1)^k\sqrt{(z-\alpha)(z-\beta)}}{2\sqrt{(z-\alpha)(z-\beta)}}, \quad k\in\{0,1\}.
\]
Then the following theorem holds.

\begin{theorem}
\label{thm:typeI}
Let the measures \( \mu_i \) be as in \eqref{weights} and polynomials \( A_\n^{(i)}(x) \) be as in \eqref{1.1} and \eqref{n_2}, \( i\in\{1,\ldots,d\} \). Further, let \( \mathcal N_\vc=\{\n\} \) be a sequence for which \eqref{multi-indices} holds. Then it holds for \( \n\in\mathcal N_\vc \) that
\[
\left\{
\begin{array}{lll}
A_\n^{(i)}(z) &=& \displaystyle -\left(1 + \mathcal O\big(|\n|^{-1}\big)\right)\frac{(\Pi_\n^{(i)}w_{\n,i})(z)}{\gamma_\n(S_\n\Phi_\n)^{(i)}(z)}, \medskip \\
A_\n^{(i)}(x) &=& \displaystyle -\left(1 + \mathcal O\big(|\n|^{-1}\big)\right)\frac{(\Pi_\n^{(i)}w_{\n,i})_+(x)}{\gamma_\n(S_\n\Phi_\n)_+^{(i)}(x)}  - \left(1 + \mathcal O\big(|\n|^{-1}\big)\right) \frac{(\Pi_\n^{(i)}w_{\n,i})_-(x)}{\gamma_\n(S_\n\Phi_\n)_-^{(i)}(x)},
\end{array}
\right.
\]
where the first relation holds uniformly on closed subsets of \( \overline{\mathbb{C}} \setminus \Delta_{\vc,i} \) and the second one holds uniformly on compact subsets of \( \Delta_{\vc,i}^\circ \), \( i\in\{1,\ldots,d\} \). Finally, let \( L_\n(z) \) be given by \eqref{Ln}. Then it holds for \( \n\in\mathcal N_\vc \) that
\begin{equation}\label{sd_asl}
L_\n(z) = \left(1 + \mathcal O\big(|\n|^{-1}\big)\right)\frac{\Pi_\n^{(0)}(z)}{\gamma_\n(S_\n\Phi_\n)^{(0)}(z)}
\end{equation}
uniformly on closed subsets of \( \overline{\mathbb{C}}\setminus \bigcup_{i=1}^d \Delta_i \).
\end{theorem}

\subsection{JMs on infinite trees for AS: asymptotics of the Green's functions} 

The asymptotics of the recurrence coefficients and polynomials of the first type can be used to compute asymptotics of the Green's functions of the operator $\cal{J}_{\vec{\kappa}}$ defined in \eqref{ooo}. We can use \eqref{sd_f4},\eqref{sd_f5}, and \eqref{muuu} to this end. For simplicity, we consider $d=2$ and suppose that $|Y|\to+\infty$ in such a way that $\vec{N}:=\Pi(Y)$ satisfy \eqref{multi-indices}. It follows from \eqref{sd_f5} and \eqref{my-gammas} that
\begin{equation}
\label{sd_sob}
G(Y,O,z)=-\left(\kappa_1\widehat\mu_2(z)/\|\mu_2\|+\kappa_2\widehat\mu_1(z)/\|\mu_1\|  \right)^{-1}\cdot \frac{L_Y(z)}{m_Y},
\end{equation}
where \( m_Y \) was defined in \eqref{muuu}. Then the asymptotics of $m_Y$ is derived from the asymptotics of the recurrence coefficients \eqref{limit} and \eqref{sd_asl} can be employed to control asymptotics of $L_Y$.

Notice that the projection of a general path from $O$ to $Y$ to the lattice $\mathbb{N}^2$ can be complicated and it can go through many intermediate ``angular'' regimes before reaching $\Pi(Y)$ which defines the terminal value of $\vec{c}$. This makes the asymptotics of $G(Y,O,z)$ very sensitive not only to $\Pi(Y)$  but also to the path itself. However, for  generic $Y$, this asymptotics takes much simpler form. Indeed, consider a random path in $\cal{T}$ that starts at $O$ and goes to infinity so that, when moving from $Y$ to $Y_{(ch),1}$ or $Y_{(ch),2}$, we chose the next vertex with equal probability. We denote the resulting path by $\{Y^{(n)}\}, \, n=0,1,\ldots$.

\begin{proposition}
With probability one, the asymptotics of $G(Y^{(n)},O,z)$ is given by
\[
G(Y^{(n)},O,z)=-\frac{1 + \mathcal O\big(|\n|^{-1}\big)}{\kappa_1\widehat\mu_2(z)/\|\mu_2\|+\kappa_2\widehat\mu_1(z)/\|\mu_1\|}\frac1{m_{Y^{(n)}}} \frac{\Pi_\n^{(0)}(z)}{\gamma_\n(S_\n\Phi_\n)^{(0)}(z)}
\]
uniformly on closed subsets of \( \overline{\mathbb{C}}\setminus(\Delta_{(\frac 12,\frac 12),1}\cup\Delta_{(\frac 12,\frac 12),2}) \), and
\[
m^{-1}_{Y^{(n)}}=(1+o(1))^n\left(  (A_{(\frac 12,\frac 12),1} A_{(\frac 12,\frac 12),2}   \right)^{n/4}.
\]
\end{proposition}
\begin{proof}
We will work with \eqref{sd_sob}. Consider $\{\Pi(Y^{(n)})\}$, the projection of the path $\{Y^{(n)}\}$. Project $\{\Pi(Y^{(n)})\}$ to the line $y+x=0$ in $\mathbb{R}^2$ denoting the resulting sequence by $\{\upsilon^{(n)}\}$. It is the standard random walk defined on the line $x+y=0$ with each step of the size $\sqrt 2$. By the law of iterated logarithm \cite{morters_peres}, we have
\[
\limsup_{n\to\infty} \left| \frac{\upsilon^{(n)}/\sqrt{2}}{\sqrt{2n\log\log n}} \right|=1
\]
with probability $1$. Therefore,  almost surely $\{\Pi(Y^{(n)})\}$ satisfies conditions \eqref{multi-indices} with $c_1=c_2=0.5$. We can use \eqref{sd_asl} to write
\[
L_{Y^{(n)}}(z) = \left(1 + \mathcal O\big(|\n|^{-1}\big)\right)\frac{\Pi_\n^{(0)}(z)}{\gamma_\n(S_\n\Phi_\n)^{(0)}(z)}
\]
uniformly on closed subsets of \( \overline{\mathbb{C}}\setminus(\Delta_{(\frac 12,\frac 12),1}\cup\Delta_{(\frac 12,\frac 12),2}) \), where \( \vec{n}=\Pi(Y^{(n)}) \). The asymptotics of the recursion coefficients \eqref{limit} yields
\[
m^{-1}_{Y^{(n)}} = (1+o(1))^n\prod_{j=1}^n  A_{(\frac 12,\frac 12),\xi_j}^{1/2}
\]
where $\xi_j=1$ if the projection of the path to $\mathbb{N}^2$ goes to the right at $j$-th step and $\xi_j=2$ if it goes up. Taking logarithm of both sides of this formula and using the law of iterated logarithm one more time gives
\[
\frac{\log m^{-1}_{Y^{(n)}}}{n}-\frac{\log \Bigl(A_{(\frac 12,\frac 12),1} A_{(\frac 12,\frac 12),2}\Bigr)}{4}\to 0
\]
with probability $1$. This proves claimed asymptotics.
\end{proof}

   \bigskip

\appendix

\section{}
\label{appA}

In this appendix, we prove theorem \ref{me1} and some auxiliary statements used in the main text. Part (A) is well-know (see, e.g., \cite{Ang19,Nik79,Ismail}). The positivity of coefficients $a_{\vec{n},j}$, i.e., condition (B), is a part of folklore but we provide the proof below anyway. Analog of (C) for the diagonal step-line recurrences was proved in \cite{ApKalLLRocha06}. Before giving the proof of the part (C) for the nearest neighbor recurrence coefficients, we list several lemmas. Some of them are well-known but we state them for completeness of the exposition. Recall that $\Delta_j$ is the smallest interval containing $\supp \,\mu_j$. Without loss of generality we assume that $\Delta_1<\Delta_2<\ldots<\Delta_d$.

\begin{lemma} \label{sd_c1}We have representations
\begin{equation}\label{por1}
a_{\vec{n},j}=\frac{\displaystyle \int_{\mathbb{R}}
P_{\vec{n}}(x)\,x^{n_j}d\mu_j(x)}{\displaystyle \int_{\mathbb{R}}
P_{\vec{n}-\vec{e}_j}(x)\, x^{n_j-1}d\mu_j(x)}, \quad \vec{n}\in \mathbb{Z}_+^d, \quad j\in \{1,\ldots,d\},\quad n_j-1\ge 0,
\end{equation}
and
\begin{equation}\label{por}
b_{\vec{n}-\vec{e}_j,j}=\int_{\mathbb{R}}
x^{|\vec{n}|}Q_{\vec{n}}(x)-\int_{\mathbb{R}}
x^{|\vec{n}|-1}Q_{\vec{n}-\vec{e}_j}(x),\quad \vec{n}\in \mathbb{N}^d, \quad j\in \{1,\ldots,d\}\,.
\end{equation}

\end{lemma}
\begin{proof}
To get \eqref{por1}, consider \eqref{1.7}, multiply it by $x^{n_j-1}$ and integrate against $\mu_j$. To prove \eqref{por}, take \eqref{1.5}, multiply  it by $x^{|\vec{n}|-1}$ and integrate over the line. Orthogonality conditions \eqref{1.1} and normalization \eqref{n_2} give \eqref{por1} and \eqref{por}.
\end{proof}

\begin{remark}
 Formula \eqref{por1} is well-known (see, e.g., \cite{Ismail}). Later in the text, we will explain why the denominator in \eqref{por1} is non-zero.\end{remark}

We will use the following  lemma. Its first claim is well-known (\cite{Ismail}, theorem
23.1.4 and \cite{w1}). 
\begin{lemma}
\label{lem:a4}
 $P_{\vec{n}}$ has $n_j$ simple zeros on $\Delta_j$, the zeros of $P_{\vec{n}+\vec{e}_m}$ and $P_{\vec{n}}$ interlace for any \( m\in\{1,\ldots,m\} \). Moreover,  let \( \{x_{\vec n+\vec e_m,i}\}_{i=1}^{|\vec n|+1} \) be the zeros of \( P_{\vec n +\vec e_m} \) labeled in the increasing order. Then
\[
x_{\vec n+\vec e_j,1} < x_{\vec n+\vec e_k,1} < x_{\vec n+\vec e_j,2} < x_{\vec n+\vec e_k,2} < \ldots < x_{\vec n+\vec e_j,|\vec n|+1} < x_{\vec n+\vec e_k,|\vec n|+1}
\]
for any \( j<k \), \( j,k\in\{1,\ldots,d\} \). That is, the zeros of \( P_{\vec n +\vec e_k} \) and \( P_{\vec n +\vec e_j} \) interlace and the zeros of  \( P_{\vec n +\vec e_k} \) dominate the ones of  \( P_{\vec n +\vec e_j} \).
\end{lemma}
\begin{proof}
We present the proof of the second claim, it can be easily adjusted to handle the first one as well. Given constants \( A,B \) such that \( |A|+|B|> 0 \), the polynomial \( AP_{\vec n +\vec e_k}(x) + BP_{\vec n +\vec e_j}(x) \) has at most \( |\vec n|+1 \) zeros and satisfies \( n_i \)  orthogonality conditions on \( \Delta_i \) for each \( i\in\{1,\ldots,d\} \). Therefore, it must have at least \( n_i \) zeros of odd multiplicity on \( \Delta_ i \) for each $i$. However, since the total number of real zeros is at most \( |\vec n|+1 \), we conclude that all of them are simple. We claim that \( P_{\vec n +\vec e_k} \) and \( P_{\vec n +\vec e_j} \) do not have a common zero. Indeed, if there were a common zero \( x_* \), then by taking \( A=P_{\vec n +\vec e_j}^\prime(x_*) \) and \( B=-P_{\vec n +\vec e_k}^\prime(x_*) \), we would obtain a polynomial with a double zero at \( x_* \) (\(|A|+|B|> 0\) holds as all the zeros of \( P_{\vec n +\vec e_k} \) and \( P_{\vec n +\vec e_j} \) are simple as well). Thus, the expression
\[
P_{\vec n +\vec e_j}(y)P_{\vec n +\vec e_k}(x) - P_{\vec n +\vec e_k}(y)P_{\vec n +\vec e_j}(x), 
\]
as a function of \(x\), vanishes at \( y \) and has only simple zeros. This implies that
\[
\det\left[\begin{matrix} P_{\vec n +\vec e_j}(y) & P_{\vec n +\vec e_k}(y) \medskip \\ P_{\vec n +\vec e_j}^\prime(y) & P_{\vec n +\vec e_k}^\prime(y) \end{matrix}\right] \neq 0
\]
for all \( y \). In a standard fashion (see, e.g., \cite{w1}, proof of theorem 2.1) this leads to the interlacing of the zeros of \( P_{\vec n +\vec e_k} \) and \( P_{\vec n +\vec e_j} \). Since \( \Delta_j<\Delta_k \) and \( P_{\vec n +\vec e_j} \) has \( n_j+1 \) zeros on \( \Delta_j \) while  \( P_{\vec n +\vec e_k} \) has \( n_j \) zeros  there, the domination property follows from interlacing.
\end{proof}

\begin{proof}[Proof of theorem~\ref{me1}: condition (B)]
According to lemma~\ref{lem:a4}, we can write
\[
P_{\vec{n}}=p_{\vec{n}}^{(1)} \cdots p_{\vec{n}}^{(d)}\,,
\]
where each polynomial $p^{(j)}_{\vec{n}}$ is monic and has $n_j$ zeros on $\Delta_j$.  Thus, we can rewrite \eqref{por1} as
\begin{equation}\label{loi}
a_{\n,j}=\frac{\displaystyle \int_{\Delta_j} \big(p_\n^{(j)}(x)\big)^2\prod_{i\neq j}p_\n^{(i)}(x)d\mu_j(x)}{\displaystyle \int_{\Delta_j} \big(p_{\n-\vec{e}_j}^{(j)}(x)\big)^2\prod_{i\neq j}p_{\n-\vec{e}_j}^{(i)}(x) d\mu_j(x)}\,.
\end{equation}
Since the products \( \prod_{i\neq j}p_\n^{(i)}(x) \) and \( \prod_{i\neq j}p_{\n-\vec e_j}^{(i)}(x) \) are non-vanishing and have the same sign on \( \Delta_j \) according to lemma~\ref{lem:a4}, the positivity of $a_{\n,j}$ follows.
\end{proof}

As before, let us write \( \Delta_j =[\alpha_j,\beta_j] \), \( j\in\{1,\ldots,d\} \). We further put \( g_i:=\alpha_{i+1}-\beta_i \), \(i\in\{1,\ldots, d-1\} \), and set $\Delta_{\max}:=[\alpha_1,\beta_d]$, $g_{\min}:=\min_i g_i$.

\begin{lemma}
\label{loi1}
Let $\n\in \mathbb{N}^d$. We have
\[
\sup_{x\in \Delta_j}\left|\frac{p_{\vec{n}}^{(m)}(x)}{p_{\vec{n}-\vec{e}_j}^{(m)}(x)}\right|\le \frac{|\Delta_{\max}|}{g_{\min}}, \quad m\neq j\,.
\]
\end{lemma}
\begin{proof}
Put $x_k=|z_{k,\vec{n},m}-x|$, $\xi_k=|z_{k,\vec{n}-\vec{e}_j,m}-x|$, where $x\in \Delta_j$ and $\{z_{k,\vec{n},m}\}$ are zeros of $p^{(m)}_{\vec{n}}$ on $\Delta_m$ and we assume that \( x_k,\xi_k \) are monotonically increasing with \( k \). It follows from lemma~\ref{lem:a4} that either \( x_i\leq \xi_i \), \( i\in\{1,\ldots,n_m\} \), in which case
\[
 \frac{x_1\cdots x_{n_m}}{\xi_1\cdots\xi_{n_m}} \leq 1 \leq \frac{|\Delta_{\max}|}{g_{\min}},
\]
or \( 0 < g_{\min} \leq \xi_1\leq x_1 \leq \xi_2 \leq x_2 \leq \ldots \leq \xi_{n_m} \leq x_{n_m}\leq \Delta_{\max} \), in which case
\begin{equation}
\label{sd_kl}
 \frac{x_1\ldots x_{n_m}}{\xi_1\ldots
\xi_{n_m}}= \left(\frac{x_1}{\xi_2} \cdots \frac{x_{n_m-1}}{\xi_{n_m}}\right)\cdot
\frac{x_{n_m}}{\xi_1} \le \frac{x_{n_m}}{\xi_1} \le \frac{|\Delta_{\max}|}{g_{\min}}. \qedhere
\end{equation}
\end{proof}

 If $\sigma$ is positive  measure on $\mathbb{R}$, denote the corresponding monic orthogonal polynomial by
$P_n(z,\sigma)$ or just $P_n(\sigma)$.

\begin{lemma}\label{li3} We have
\[
\|P_n(\sigma)\|_{\sigma}^2=\min_{Q:\, \deg Q=n, Q {\,\rm is\, monic}}\|Q\|_{\sigma}^2\,.
\]

\end{lemma}
\begin{proof}
This follows from the orthogonality conditions.
\end{proof}

\begin{lemma} 
\label{li4}
Let $\sigma$ be positive measure on $\mathbb{R}$ with compact support. Set $\Delta:=\mathrm{Ch}(\supp\,\sigma)$. Then
\[
\quad \sup_{n\in \mathbb{Z}_+} \frac{\|P_{n+1}(\sigma)\|^2_{\sigma}}{\|P_{n}(\sigma)\|_{\sigma}^2}\le (|\Delta|/2)^2\,.
\]
\end{lemma}
\begin{proof} 
See the explanation between \eqref{1.17} -- \eqref{lope}.
\end{proof}

Define
\[
d\sigma_\n^{(j)}:=\Bigl(\prod_{m\neq j}\big|p_{\vec{n}}^{(m)}\big|\Bigr)\cdot d\mu_j\,.
\]
Then the following lemma trivially holds.

\begin{lemma} 
\label{li5}
Polynomial $p_\n^{(j)}$ is \( n_j \)-th monic orthogonal polynomial with respect to the measure $\sigma_\n^{(j)}$, i.e.,
$p_\n^{(j)}=P_{n_j}\big(\sigma_\n^{(j)}\big)$.
\end{lemma}

For the proof of next lemma, see, e.g., \cite{NikishinSorokin}, p. 135, Proposition 3.4 and \cite{peralta}, Proposition 2.2.
\begin{lemma}
\label{li6}
$A_{\vec{n}}^{(j)}$ has $n_{j}-1$ simple zeros on $\Delta_j$.
\end{lemma}

Denote by \( \kappa_\n \) the product of the leading coefficients of the polynomials $A_{\vec{n}}^{(j)}$ and define
 \[
M_\n(z):=\kappa_\n^{-1}\prod_{j=1}^d A_\n^{(j)}(z) = z^{|\n|-d} + \mathcal O\big(z^{|\n|-d-1} \big) \quad \text{as} \quad z\to\infty.
\]

\begin{lemma}
\label{li7}
Given \( \n\in\N^d \), there exists a polynomial $D_\n(x)=\prod_{i=1}^{d-1}(x-\xi_{\n,i})$, where \( \xi_{\n,i}\in\{\alpha_i,\beta_i\} \), such that \begin{equation}
\label{sd_g6}
\sum_{j=1}^d \int_{\Delta_j} \left|A_{\vec{n}}^{(j)}D_\n M_\n\right| d\mu_j=1\,.
\end{equation}
\end{lemma}
\begin{proof}
Since \( D_\n M_\n \) is a monic polynomial of degree \( |\n|-1 \), we get from orthogonality conditions \eqref{1.1} and normalization \eqref{n_2} that
\[
\sum_{j=1}^d \int_{\Delta_j} A_\n^{(j)}D_\n M_\n d\mu_j = \int_\R D_\n M_\n Q_\n =1\,.
\]
Notice first that the polynomial $ A_\n^{(j)}D_\n M_\n$ does not change its sign on $\Delta_j$ for any choice of  \( \xi_{\n,i}\in\{\alpha_i,\beta_i\} \), \( i\in\{1,\ldots,d-1\} \). To prove the lemma, choose \( \xi_{\n,i}\in\{\alpha_i,\beta_i\} \), starting with \( i=d-1 \) and continuing down to \( i=1 \), so that \( A_\n^{(j)}D_\n M_\n \) has the same sign on \( \Delta_j \) as \( A_\n^{(d)}D_\n M_\n \) has on \( \Delta_d \) (the latter is necessarily positive).
\end{proof}

The next lemma follows from the proof of theorem 5 in \cite{peralta} (see also \cite{w1}).
\begin{lemma}
\label{li8}
The zeros of $A_{\vec{n}}^{(j)}$ and $A_{\vec{n}+\vec{e}_l}^{(j)}$ interlace for any \( l\in\{1,\ldots,d \} \).
\end{lemma}

\medskip

\begin{proof}[Proof of theorem~\ref{me1}: condition (C)] {\bf The first bound in (C) from \eqref{1.9}.} By \eqref{loi} and lemma~\ref{loi1}, we get that
\[
a_{\vec{n},j} \le  \left(\frac{|\Delta_{\max}|}{g_{\min}}\right)^{d-1}   \frac{\displaystyle \int_{\mathbb{R}} \big(p_\n^{(j)}\big)^2d\sigma_\n^{(j)}}{\displaystyle \int_{\mathbb{R}} \big(p_{\n-\vec{e}_j}^{(j)}\big)^2d\sigma_\n^{(j)}}\,.
\]
Then, it follows from lemmas \ref{li5}, \ref{li4}, and \ref{li3}, that
\[
\frac{\displaystyle \int_{\mathbb{R}} \big(p_\n^{(j)}\big)^2d\sigma_\n^{(j)}}{\displaystyle \int_{\mathbb{R}} \big(p_{\n-\vec{e}_j}^{(j)}\big)^2d\sigma_\n^{(j)}} = 
\frac{\displaystyle \int_{\mathbb{R}} \big(P_{n_j}\big(\sigma_\n^{(j)}\big)\big)^2d\sigma_\n^{(j)}}{\displaystyle \int_{\mathbb{R}} \big(p_{\n-\vec{e}_j}^{(j)}\big)^2d\sigma_\n^{(j)}} \le
\left(\frac{|\Delta_j|}{2}\right)^2 \frac{\displaystyle \int_{\mathbb{R}} \big(P_{n_j-1}\big(\sigma_\n^{(j)}\big)\big)^2d\sigma_\n^{(j)}}{\displaystyle \int_{\mathbb{R}} \big(p_{\n-\vec{e}_j}^{(j)}\big)^2d\sigma_\n^{(j)}} \leq \left(\frac{|\Delta_j|}{2}\right)^2.
\]
Thus,
\[
\sup_{\n\in \mathbb{N}^d} a_{\n,j}\le \left(\frac{|\Delta_{\max}|}{g_{\min}}\right)^{d-1} \left(\frac{|\Delta_j|}{2}\right)^2\,.
\]

\smallskip

\noindent
{\bf The second bound in (C) from \eqref{1.9}.} It follows from \eqref{por} that
\[
b_{\n,j} = Y_{\n+\vec e_j} - Y_\n, \quad Y_{\vec{n}}:= \int_{\mathbb{R}} x^{|\vec{n}|}Q_\n(x).
\]
Put
\[
\eta_\n := \int_{\mathbb{R}} x(M_\n D_\n)(x)Q_\n(x) =\sum_{j=1}^d  \int_{\Delta_j}x (A_\n^{(j)}D_\n M_\n)(x)d\mu_j(x)\,.
\]
Orthogonality conditions \eqref{1.1} and normalization \eqref{n_2} yield that 
\[
\eta_\n = Y_{\vec{n}}+C_\n + \widetilde C_\n\,,
\]
where $C_\n$ is defined by $M_\n(x)=x^{|\vec{n}|-d}+C_\n x^{|\vec{n}|-d-1}+\cdots$ and  $\widetilde C_\n$ is defined by $D_\n(x)=x^{d-1}+\widetilde C_\n x^{d-2}+\cdots$. It follows from \eqref{sd_g6} that
\[
\sup_{\n\in \mathbb{Z}_+^d} |\eta_{\vec{n}}|\le \sup_{x\in \Delta_{\max}}|x|.
\]
Furthermore, since each \( \xi_{\n,i}\in\{\alpha_i,\beta_i\} \), \( i\in\{ 1,\ldots,d-1 \} \), we have that
\[
\big |\widetilde C_{\n+\vec e_j} - \widetilde C_\n\big|=\left |\sum_{i=1}^{d-1}\xi_{\n+\vec e_j,i} - \sum_{i=1}^{d-1}\xi_{\n,i}\right| \leq \sum_{i=1}^{d-1} \big|\xi_{\n+\vec e_j,i} - \xi_{\n,i} \big| \leq |\Delta_{\max}|.
\]
Finally, if we denote the zeros of \( M_\n \) by $\{x_{\n,i}\}_{i=1}^{|\n|-d}$ in the increasing order, it holds that 
\[
\big|C_{\n+\vec e_j}-C_\n\big| = \left| \sum_{i=1}^{|\n|-d+1} x_{\n+\vec e_j,i} - \sum_{i=1}^{|\n|-d} x_{\n,i} \right| \leq |\Delta_{\max}| + \sup_{x\in \Delta_{\max}}|x|
\]
by lemmas~\ref{li6} and \ref{li8}. Since \( |\Delta_{\max}| \leq 2 \sup_{x\in \Delta_{\max}}|x| \), we have that
\[
\sup_{\n\in \N^d,j\in \{1,\ldots,d\}} |b_{\n,j}|\le 7\sup_{x\in \Delta_{\max}}|x|. \qedhere
\]
\end{proof}

\begin{remark}
The arguments we have given above imply that
\[
\sup \limits_{\n\in \Z_+^d,j\in \{1,\ldots,d\}}a_{\n,j}<\infty\,,\,  \sup \limits_{\n \in \mathbb{Z}_+^d,j\in \{1,\ldots,d\}}|b_{\vec{n},j}|<\infty \,,
\]
that is, we can replace \( \N \) with \( \Z_+ \) in \eqref{1.9}. Indeed, consider all $\{a_{\vec{n}, l}\}$ and    $\{b_{\vec{n}, l}\}$ for which at least one coordinate in $\vec{n}$ is zero. Among them, we first take those $\vec{n}$ for which exactly one component, say $n_j$, in $\vec{n}$ is equal to zero and $l=j$. For that family, the boundedness of $\{b_{\vec{n},l}\}$ has been proven in the above theorem and recall that $a_{\vec{n},l}=0$ for such indices.
  To prove uniform estimate for other coefficients, we can argue by induction in $d$. Indeed, for $d=2$, the recurrence coefficients evaluated on the margins are uniformly bounded because they are recurrence coefficients of one-dimensional Jacobi matrices with compactly supported measures of orthogonality. For general $d$, we notice that the polynomials of first/second type with indices on the margin are in fact the polynomials of the first/second type with respect to $d-1$ orthogonality measures in Angelesco system and we can argue by induction. 
\end{remark}

\begin{remark}
We want to give another proof of the uniform estimate of $b_{\vec{n},j}$. This argument is taken from \cite{ApKalLLRocha06}. Divide recursion \eqref{1.7} by $xP_\n(x)$ and integrate over the contour $\Gamma$ which encircles $\{0\}\cup_{j=1}^d \Delta_j$ to get
\[
\frac{1}{2\pi\ic}\int_{\Gamma}\left( 1-\frac{P_{\n+\vec{e}_j}(z)}{z P_\n(z)}  \right)dz=b_{\vec{n},j}+\frac{1}{2\pi \ic}\int_{\Gamma}\sum_{l=1}^d a_{\vec{n},l}\frac{P_{\n-\vec{e}_l}(z)}{zP_\n(z)}dz\,.
\]
The last term is zero by residue calculus at infinity. Using the interlacing property of zeros, we can write
\[
|b_{\vec{n},j}|\le \frac{1}{2\pi }\int_{\Gamma}\left| \frac{P_{\n+\vec{e}_j}(z)}{z P_\n(z)}  \right||dz|\le   C_{\Delta_{\max}}\,,
\]
where one needs to use a variation of \eqref{sd_kl}.
\end{remark}

The next lemma shows that the coefficients $\{b_{\vec{n},j}\}$ in fact are monotonic in $j$. 

\begin{lemma}\label{sd_monot} 
For all \( \vec n\in\mathbb Z_+^d \) and any \( j<k \), \( j,k\in\{1,\ldots,d\} \) it holds that
\[
b_{\vec n,j} < b_{\vec n,k}.
\]
\end{lemma}
\begin{proof}
It follows from the recurrence relations that
\[
\big(b_{\vec n,j} - b_{\vec n,k}\big)P_{\vec n}(x) = P_{\vec n + \vec e_k}(x) -  P_{\vec n + \vec e_j}(x).
\]
Since \( P_{\vec n}(x) \) is a monic polynomial and the second coefficient of any monic polynomial is minus the sum of its zeros, we have that
\[
b_{\vec n,j} - b_{\vec n,k} = \sum_{i=1}^{|\vec n|+1} \big( x_{\vec n+\vec e_j,i}-x_{\vec n+\vec e_k,i}\big),
\]
where \( \{x_{\vec n+\vec e_m,i}\} \) are the zeros of \( P_{\vec n +\vec e_m} \) labeled in the increasing order. The claim now follows from the second claim in lemma~\ref{lem:a4}.
\end{proof}



 \begin{lemma} 
 Suppose $\vec{\mu}$ defines an Angelesco system and  $\Delta_{\max}\subset [-R,R]$. Then
\begin{equation}\label{sd_f2}
|L_n(z)|\le  (|z|-R)^{-|\vec{n}|}, \quad |z|>R\,.
\end{equation}
\end{lemma}
\begin{proof}
Recall that $A_\n^{(j)}$ has $n_j-1$ simple zeros on $\Delta_j$. Let us abbreviate  \( O_\n=M_\n D_\n \), see lemma~\ref{li7}, and write
\[
\int_{\mathbb{R}} \frac{Q_\n(x)}{x-z}=\frac{1}{O_\n(z)}\int_{\mathbb{R}}
\frac{(O_\n(z)-O_\n(x))Q_\n(x)}{x-z}+\frac{1}{O_\n(z)}\int_{\mathbb{R}}
\frac{O_\n(x)Q_\n(x)}{x-z}\,.
\]
Since $(O_\n(z)-O_\n(x))/(z-x)$ is a polynomial in $x$ of degree $|\vec{n}|-2$, the first summand of the left-hand side of the equality above is zero. For the second one, we can write
\begin{multline*}
\left|\frac{1}{O_\n(z)}\int_{\mathbb{R}} \frac{O_\n(x)Q_\n(x)}{x-z}\right|=\left|\frac{1}{O_\n(z)} \int_{\mathbb{R}}
\frac{\sum_{j=1}^d \int_{\Delta_j} (A_\n^{(j)}D_\n M_\n)(x) d\mu_j(x)}{x-z}\right| \le \\
\frac{1}{|O_\n(z)|} \int_{\mathbb{R}} \frac{\sum_{j=1}^d \int_{\Delta_j} \big|(A_\n^{(j)}D_\n M_\n)(x)\big|d\mu_j(x)  }{|x-z|} \le
\frac{1}{(|z|-R)^{|\vec{n}|}},\, \quad |z|>R,
\end{multline*}
due to \eqref{sd_g6}.
\end{proof}

\section{}

\label{ApB}

\subsection{Strong asymptotics of MOPs}

Let $\big(\widehat\mu_1,\ldots,\widehat\mu_d\big)$ be a vector of Markov functions of the measures \( \mu_i \), that is,
\[
\widehat\mu_i(z) := \int\frac{d\mu_i(x)}{z-x} = \frac1{2\pi\ic}\int_{\Delta_i}\frac{\rho_i(x)}{x-z}\:d x, \qquad z\in\overline{\mathbb{C}}\setminus \Delta_i.
\]
The above definition explains the somewhat perplexing normalization in \eqref{weights} as \( (\widehat\mu_{i+} - \widehat\mu_{i-})(x) = \rho_i(x) \), \( x\in \Delta_i^\circ \), by Sokhotski-Plemelj formulae. Further, let linearized error function \( R_\n^{(i)} \) be given by \eqref{Rnj} and polynomials \( P_\n^{(i)} \) be given by \eqref{Pnj}.

\begin{theorem}
\label{cor:typeIIb}
Under the conditions of theorem~\ref{thm:typeII}, it holds for all \( \n\in\mathcal N_\vc \) that
\[
\left( \widehat\mu_i - \frac{P_\n^{(i)}}{P_\n}\right)(z) = \frac{R_\n^{(i)}(z)}{P_\n(z)} = \frac{1 + \mathcal O\big(|\n|^{-1}\big)}{w_{\n,i}(z)}\frac{\big(S_\n\Phi_\n\big)^{(i)}(z)}{\big(S_\n\Phi_\n\big)^{(0)}(z)},
\]
\( i\in\{1,\ldots,d\} \), uniformly on closed subsets of \( \overline{\mathbb{C}} \setminus \bigcup_{i=1}^d \Delta_{\vc,i} \).
\end{theorem}

Clearly, the error of approximation is small in \( D_{\n,i}^+ \) and is large in \( D_{\n,i}^- \), where
\[
D_{\n,i}^+ := \big\{z\in\mathbb{C}:\big|\Phi_\n^{(i)}(z)/\Phi_\n^{(0)}(z)\big|<1\big\} \quad \text{and} \quad D_{\n,i}^- := \big\{z\in\mathbb{C}:\big|\Phi_\n^{(i)}(z)/\Phi_\n^{(0)}(z)\big|>1\big\}.
\]
It is known \cite{GRakh81,Y16} that the domains \( D_{\n,i}^\pm \) converge in the Hausdorff metric to certain domains \( D_{\vc,i}^\pm \) when \( |\n|\to\infty \), \( \n\in\mathcal N_\vc \). The divergence domain \( D_{\vc,i}^- \) is always bounded, possibly empty, and necessarily contains \( \Delta_i\setminus \Delta_{\vc,i} \), see Figure~\ref{fig:pushing}. The ratio \( |\Phi_\n^{(i)}/\Phi_\n^{(0)}| \) is geometrically small on closed subsets of \( D_{\vc,i}^+ \).

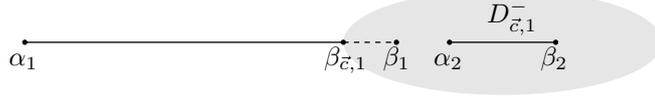
\begin{figure}[!ht]
\centering
\includegraphics[scale=.55]{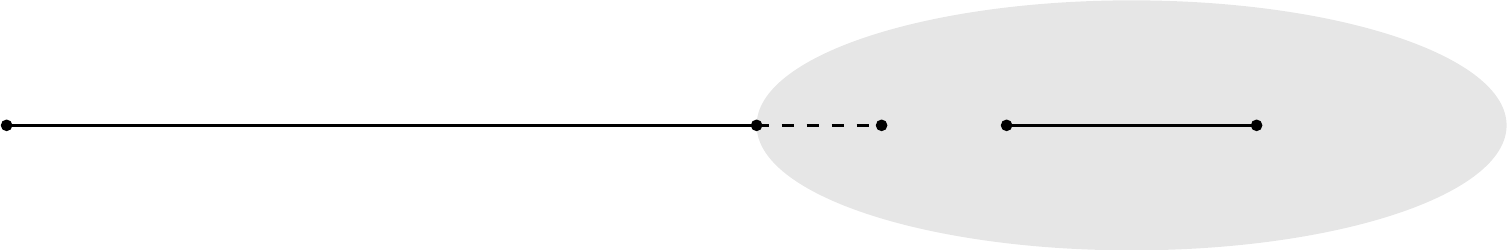}
\caption{\small Schematic representation of the pushing effect and a divergence domain in the case of 2 intervals (in this case $D_{\vc,2}^-=\varnothing$).}
\begin{picture}(0,0)
\put(-123,54){$\alpha_1$}
\put(-5,54){$\beta_{\vc,1}$}
\put(17,54){$\beta_1$}
\put(36,54){$\alpha_2$}
\put(76,54){$\beta_2$}
\put(56,70){$D_{\vc,1}^-$}
\end{picture}
\label{fig:pushing}
\end{figure}

\begin{proof}[Proof of theorems~\ref{thm:typeII} and~\ref{cor:typeIIb}]

Theorem~\ref{thm:typeII}  and corollary~\ref{cor:typeIIb} were proven in \cite[Theorem~2.5]{Y16}.  Extension to multiple orthogonal polynomials \cite{GerKVA01} of by now classical approach of Fokas, Its, and Kitaev \cite{FIK91,FIK92} connecting orthogonal polynomials to matrix Riemann-Hilbert problems was used followed by the asymptotic analysis based on the non-linear steepest descent method of Deift and Zhou \cite{DZ93}. The following definitions will be important for the remaining proofs in this section. Set
\begin{equation}
\label{Y}
\boldsymbol Y_\n := \left(\begin{array}{cccc}
P_{\vec n} & R_{\vec n}^{(1)} & \cdots & R_{\vec n}^{(d)} \smallskip\\
m_{\vec n,1}P_{\vec n-\vec e_1} & m_{\vec n,1}R_{\vec n-\vec e_1}^{(1)} & \cdots & m_{\vec n,1}R_{\vec n-\vec e_1}^{(d)} \\
\vdots & \vdots & \ddots & \vdots \\
m_{\vec n,d}P_{\vec n-\vec e_d} & m_{\vec n,d}R_{\vec n-\vec e_d}^{(1)} & \cdots & m_{\vec n,d}R_{\vec n-\vec e_d}^{(d)}
\end{array}\right),
\end{equation}
where the constants $m_{\vec n,i}$ are such that \( \lim_{z\to\infty}m_{\vec n,i}R_{\vec n-\vec e_i}^{(i)}(z)z^{n_i}=1 \). Further, let \( \chi_\n(\z) \) be given by \eqref{chi} on \( \RS_\n \) and the constants \( \{A_{\n,i},B_{\n,i} \}_{i=1}^d \) be as in \eqref{ABs}. Define
\begin{equation}\label{sd_jk1}
\Upsilon_{\n,i}(\z) := A_{\n,i}/\big(\chi_\n(\z)-B_{\n,i}\big), \quad i\in\{1,\ldots,d\}.
\end{equation}
Clearly, it holds that
\begin{equation}
\label{Upsilons}
\Upsilon_{\n,i}^{(i)}(\z) = z + \mathcal O(1) \quad \text{and} \quad \Upsilon_{\n,i}^{(0)}(\z) = A_{\n,i}\big( z^{-1} + B_{\n,i}z^{-2} + \mathcal O\big(z^{-3}\big) \big)
\end{equation}
as \( z\to\infty \). Let
\begin{equation}
\label{M}
\boldsymbol M_\n:=\left(\begin{array}{cccc}
S_\n^{(0)} & S_\n^{(1)}/w_{\n,1} & \cdots & S_\n^{(d)}/w_{\n,d} \medskip\\
\big( S_\n\Upsilon_{\n,1}\big)^{(0)} & \big( S_\n\Upsilon_{\n,1}\big)^{(1)}/w_{\n,1} & \cdots &\big( S_\n\Upsilon_{\n,1}\big)^{(d)}/w_{\n,d} \smallskip \\
\vdots & \vdots & \ddots & \vdots \\
\big( S_\n\Upsilon_{\n,d}\big)^{(0)} & \big( S_\n\Upsilon_{\n,d}\big)^{(1)}/w_{\n,1} & \cdots & \big( S_\n\Upsilon_{\n,d}\big)^{(d)}/w_{\n,d}
\end{array}\right)
\end{equation}
and \( \boldsymbol C_\n \) be the diagonal matrix of constants such that
\[
\lim_{z\to\infty}\boldsymbol C_\n(\boldsymbol M_\n \boldsymbol D_\n)(z)z^{-\sigma(|\n|)} = \boldsymbol I,  \quad \boldsymbol D_\n :=\diag\big(\Phi_{\vec n}^{(0)},\ldots,\Phi_{\vec n}^{(d)}\big),
\]
where $\displaystyle \sigma(\vec n) := \diag\left(|\vec n|, -n_1,\ldots,-n_d\right)$.
\begin{figure}[!ht]
\centering
\includegraphics[scale=1]{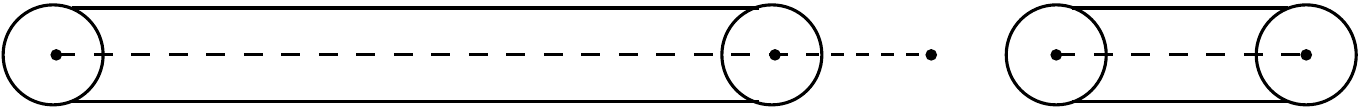}
\caption{\small Contour \( \Sigma \) (solid lines) in the case of two intervals \( \Delta_{\vc,1}=[\alpha_1,\beta_{\vc,1}] \) and \( \Delta_{\vc,2}=[\alpha_2,\beta_2] \).}
\begin{picture}(0,0)
\put(70,67){$\Sigma$}
\put(-184,42){$\alpha_1$}
\put(22,42){$\beta_{\vc,1}$}
\put(70,42){$\beta_1$}
\put(106,42){$\alpha_2$}
\put(178,42){$\beta_2$}
\end{picture}
\label{fig:Sigma}
\end{figure}
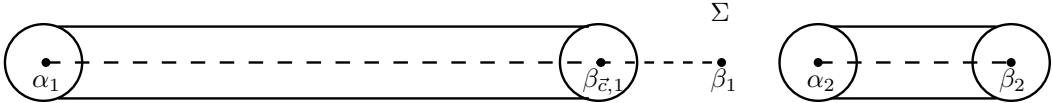
 Then it was shown in the proof of \cite[Theorem~2.5]{Y16} that there exists a contour \( \Sigma \), see Figure~\ref{fig:Sigma}, that can be made to avoid any given closed set \( K\subset \overline{\mathbb{C}}\setminus\bigcup_{i=1}^d\Delta_{\vc,i} \) except for the part \( K\cap\bigcup_{i=1}^d(\Delta_i\setminus \Delta_{\vc,i} \)) and any given compact set \( F\subset\bigcup_{i=1}^d \Delta_{\vc,i}^\circ \) such that
\begin{equation}
\label{Y1}
\boldsymbol Y_\n = \boldsymbol C_\n \boldsymbol Z_\n \boldsymbol M_\n \boldsymbol D_\n \quad \text{and} \quad \boldsymbol Y_{\n\pm} = \boldsymbol C_\n \boldsymbol Z_\n \boldsymbol M_{\n\pm} \boldsymbol D_{\n\pm}(\boldsymbol I \pm(1/\rho_l) \boldsymbol E_{l+1,1} )
\end{equation}
on \( K \) and on \( F\cap \Delta_{\vc,l} \), \( l\in\{1,\ldots,d\} \), respectively, where \( \boldsymbol Z_\n \) is holomorphic in \( \overline{\mathbb{C}}\setminus \Sigma \), all but \( (i+1) \)-st  column of \( \boldsymbol Z_\n \) are holomorphic across \( \Delta_i\setminus \Delta_{\vc,i} \), \( \boldsymbol Z_\n(\infty) = \boldsymbol I \), and \( \boldsymbol Z_\n = \boldsymbol I + \boldsymbol{\mathcal O}\big(|\n|^{-1}\big) \) uniformly in \( \overline{\mathbb{C}} \). Let \( Z_k := [\boldsymbol Z_\n]_{1,k+1} -\delta_{0k} \), where  \( \delta_{ij} \) is the usual Kronecker symbol, \( k\in\{0,\ldots,d\} \). Then we get that
\[
[\boldsymbol Z_\n \boldsymbol M_\n]_{1,1} = S_\n^{(0)}\left(1+\sum_{l=0}^dZ_l\Upsilon_{\n,l}^{(0)}\right) =  \left(1 + \mathcal O\big(|\n|^{-1}\big)\right)S_\n^{(0)}
\]
uniformly on \( K \), where the second equality holds because \(  Z_k = O\big(|\n|^{-1}\big) \) uniformly in \( \overline{\mathbb{C}} \) (including the traces on \( \Sigma \)) and the functions \( \Upsilon_{\n,l}^{(0)} \) converge to the functions \( \Upsilon_{\vc,l}^{(0)} \) also uniformly on \( K \). This proves the first asymptotic formula of theorem~\ref{thm:typeII}. Similarly, we have that
\[
[\boldsymbol Z_\n \boldsymbol M_\n]_{j+1,1} = \left(1 + \mathcal O\big(|\n|^{-1}\big)\right)\big(\Upsilon_{\n,j}S_\n\big)^{(0)}
\]
uniformly on closed subsets of \( \overline{\mathbb{C}}\setminus\bigcup_{i=1}^d\Delta_{\vc,i} \), \( j\in\{1,\ldots,d\} \). Therefore, it follows from \eqref{Upsilons} and the choice of \( \gamma_\n \) that
\begin{equation}
\label{mns}
m_{\n,j} = \left(1 + \mathcal O\big(|\n|^{-1}\big)\right)A_{\n,j}[\boldsymbol C_\n]_{j+1,j+1}\gamma_\n^{-1},
\end{equation}
\( j\in\{1,\ldots,d\} \). Furthermore, it holds that
\[
[\boldsymbol Z_\n \boldsymbol M_\n]_{1,i+1} = S_\n^{(i)}\left(1+\sum_{l=0}^dZ_l\Upsilon_{\n,l}^{(i)}\right)/w_{\n,i} =  \left(1 + \mathcal O\big(|\n|^{-1}\big)\right)S_\n^{(i)}/w_{\n,i}
\]
uniformly on closed subsets of \( \overline{\mathbb C}\setminus\Delta_{\vc,i} \), \( i\in\{1,\ldots,d\} \), where one needs to observe that even though \( \Upsilon_{\n,i}^{(i)} \) has a pole  at infinity, \( Z_i \) has a zero there, and therefore the desired estimate is obtained via the maximum modulus principle for holomorphic functions. Thus,
\begin{equation}
\label{RniAs}
R_{\n,i}(z) = \left(1 + \mathcal O\big(|\n|^{-1}\big)\right)\gamma_\n\big(S_\n\Phi_\n\big)^{(i)}(z)/w_{\n,i}(z)
\end{equation}
uniformly on closed subsets of \( \overline{\mathbb C}\setminus\Delta_{\vc,i} \), \( i\in\{1,\ldots,d\} \),  which proves Theorem~\ref{cor:typeIIb}. Finally, we get on \( F\subset \Delta_{\vc,k}^\circ \) that
\begin{eqnarray*}
P_\n &=& \gamma_\n\big(S_\n^{(0)}\Phi_\n^{(0)}\big)_\pm\left(1+\sum_{l=0}^dZ_l\Upsilon_{\n,l\pm}^{(0)}\right) \pm \gamma_\n \big(S_\n^{(k)}\Phi_\n^{(k)}\big)_\pm \left(1+\sum_{l=0}^dZ_l\Upsilon_{\n,l\pm}^{(k)}\right)/(\rho_lw_{\n,k\pm}) \\
&=& \gamma_\n\big(S_\n^{(0)}\Phi_\n^{(0)}\big)_\pm\left(1+\sum_{l=0}^dZ_l\Upsilon_{\n,l\pm}^{(0)}\right) + \gamma_\n\big(S_\n^{(0)}\Phi_\n^{(0)}\big)_\mp\left(1+\sum_{l=0}^dZ_l\Upsilon_{\n,l\mp}^{(0)}\right)
\end{eqnarray*}
by the properties of the functions \( S_\n^{(l)} \) and since \( F^{(0)}_\pm = F^{(l)}_\mp \) on \( \Delta_{\n,l} \) for a rational function \( F \) on \( \RS_\n \). This proves the second asymptotic formula of theorem~\ref{thm:typeII}.
\end{proof}

\begin{proof}[Proof of theorem~\ref{thm:typeI}]

We can decompose matrix \( \boldsymbol M_\n \) in \eqref{M} as
\[
\boldsymbol M_\n = \boldsymbol \Upsilon_\n\boldsymbol S_\n, \quad \boldsymbol S_\n := \diag\left(S_\n^{(0)},S_\n^{(1)}/w_{\n,1},\ldots,S_\n^{(d)}/w_{\n,d}\right), \quad [\boldsymbol\Upsilon_\n]_{l+1,k+1} = \Upsilon_{\n,l}^{(k)},
\]
where for convenience we put \( \Upsilon_{\n,0} \equiv 1 \). Let \( \Pi_{\n,i}(\z) \), \( i\in\{1,\ldots,d\} \), be a rational function on \( \RS_\n \) with zero/pole divisor and normalization given by
\[
\infty^{(0)} + 2\big(\infty^{(1)} + \cdots + \infty^{(d)} \big) - \infty^{(i)} - \mathcal D_\n \quad \text{and} \quad \Pi_{\n,i}^{(i)}(z) = z^{-1} + \mathcal O\big(z^{-2}\big).
\]
Observe that \( \Pi_{\n,i} = g_{\n,i}\Pi_\n\Upsilon_{\n,i} \) for some normalizing constants \( g_{\n,i} \), \( i\in\{1,\ldots,d\} \). Set \( \boldsymbol \Pi_\n \) to be the matrix such that \( \big[\boldsymbol \Pi_\n \big]_{k+1,j+1} = \Pi_{\n,j}^{(k)} \), where we put \( \Pi_{\n,0}:=\Pi_\n \). Then it holds that
\[
\big[\boldsymbol\Upsilon_\n \boldsymbol \Pi_\n \big]_{l+1,j+1} = \sum_{k=0}^d\big(\Upsilon_{\n,l}\Pi_{\n,j}\big)^{(k)},
\]
which is necessarily a meromorphic function on \( \overline{\mathbb{C}} \). As it can have at most square root singularities at the points \( \{\alpha_{\n,i},\beta_{\n,i}\} \), \( i\in\{1,\ldots,d\} \), it is a polynomial. It is further clear from the behavior of this function at infinity that \( \big[\boldsymbol\Upsilon_\n \boldsymbol \Pi_\n \big]_{l+1,j+1} \equiv \delta_{lj} \). That is,
\begin{equation}
\label{M-inverse}
\boldsymbol M_\n^{-1} = \boldsymbol S_\n^{-1}\boldsymbol\Pi_\n.
\end{equation}

Similarly to the matrix \( \boldsymbol Y_\n \), define
\[
\widehat{\boldsymbol Y}_\n := \left(\begin{matrix} L_\n & -A_\n^{(1)} & \cdots & -A_\n^{(d)} \medskip \\ -d_{\n,1} L_{\n+\vec e_1}  & d_{\n,1} A_{\n+\vec e_1}^{(1)} & \cdots & d_{\n,1}A_{\n+\vec e_1}^{(d)}\medskip \\ \vdots & \vdots & \ddots & \vdots \medskip \\ -d_{\n,d} L_{\n+\vec e_d}  & d_{\n,d} A_{\n+\vec e_d}^{(1)} & \cdots &  d_{\n,d}A_{\n+\vec e_d}^{(d)} \end{matrix}\right),
\]
where the constant \( d_{\n,i} \) is chosen so that the polynomial \( d_{\n,i}A_{\n+\vec e_i}^{(i)} \) is monic. It was shown in \cite[Theorem~4.1]{GerKVA01} that
\[
\widehat{\boldsymbol Y}_\n = \big(\boldsymbol Y_\n^\mathsf{T}\big)^{-1}.
\]
Hence, it follows from \eqref{Y1} and \eqref{M-inverse} that on closed subsets of \( \overline{\mathbb{C}}\setminus \bigcup_{i=1}^d\Delta_{\vc,i} \) it holds that
\[
\widehat{\boldsymbol Y}_\n = \boldsymbol C_\n^{-1}\widehat{\boldsymbol Z}_\n\boldsymbol \Pi_\n^\mathsf{T}\boldsymbol S_\n^{-1}\boldsymbol D_\n^{-1},
\]
where \( \widehat{\boldsymbol Z}_\n := \big(\boldsymbol Z_\n^{-1}\big)^{\mathsf T} = \boldsymbol I + \big[\widehat Z_{l,j}\big]_{l,j=1}^{d+1} \) with \( \widehat Z_{l,j}(\infty) =0 \) and \( \widehat Z_{l,j} = \mathcal O(\big|\n|^{-1}\big) \) uniformly in \( \overline{\mathbb{C}} \). Then
\[
\big[\widehat{\boldsymbol Z}_\n\boldsymbol\Pi_\n^{\mathsf T}\big]_{1,1} = \Pi_\n^{(0)} + \sum_{k=0}^d\widehat Z_{1,k+1}\Pi_{\n,k}^{(0)} = \left( 1 + \sum_{k=0}^dg_{\n,k}\widehat Z_{1,k+1}\Upsilon_{\n,k}^{(0)} \right) \Pi_\n^{(0)} \\ = \big( 1+\mathcal O\big(|\n|^{-1}\big) \big) \Pi_\n^{(0)}
\]
uniformly on closed subsets of \( \overline{\mathbb{C}}\setminus \bigcup_{i=1}^d\Delta_{\vc,i} \), where the last equality follows from the fact that the functions \( \Upsilon_{\n,i}^{(0)} \) converge to \( \Upsilon_{\vc,i}^{(0)} \) uniformly on \( \overline{\mathbb{C}}\setminus\bigcup_{i=1}^d\{\alpha_{\vc,i},\beta_{\vc,i}\} \) (including the traces on \( \bigcup_{i=1}^d\Delta_{\vc,i}^\circ \)) and the constants \( g_{\n,i} \) converge to some constants \( g_{\vc,k} \). Therefore, the last claim of the theorem follows. Similarly we get that
\[
\big[\widehat{\boldsymbol Z}_\n\boldsymbol\Pi_\n^{\mathsf T}\big]_{l+1,1} = \left(1 + \sum_{k=0}^d\widehat Z_{l+1,k+1}\frac{g_{\n,k}\Upsilon_{\n,k}^{(0)}}{g_{\n,l}\Upsilon_{\n,l}^{(0)}}\right) \Pi_{\n,l}^{(0)} = \big( 1+\mathcal O\big(|\n|^{-1}\big) \big)\Pi_{\n,l}^{(0)}
\]
uniformly on closed subsets of \( \overline{\mathbb{C}}\setminus \bigcup_{i=1}^d\Delta_{\vc,i} \), \( l\in\{1,\ldots,d\} \). Since \( L_{\n+\vec e_l}(z) = z^{-|\n|-1} + \mathcal O\big(z^{-|\n|-2}\big) \) as \( z\to\infty \), we also get that
\begin{equation}
\label{Lnl}
L_{\n+\vec e_l}(z) = \left(1 + \mathcal O\big(|\n|^{-1}\big)\right) \frac{\Upsilon_{\n,l}^{(0)}(z)}{A_{\n,l}} \frac{\Pi_\n^{(0)}(z)}{\gamma_\n(S_\n\Phi_\n)^{(0)}(z)}
\end{equation}
uniformly on closed subsets of \( \overline{\mathbb{C}}\setminus \bigcup_{i=1}^d\Delta_{\vc,i} \), \( l\in\{1,\ldots,d\} \). It further holds that
\[
\big[\widehat{\boldsymbol Z}_\n\boldsymbol\Pi_\n^{\mathsf T}\big]_{1,l+1} = \Pi_\n^{(l)} + \sum_{k=0}^d\widehat Z_{1,k+1}\Pi_{\n,k}^{(l)} = \big( 1+\mathcal O\big(|\n|^{-1}\big) \big) \Pi_\n^{(l)}
\]
uniformly on closed subsets of \( \overline{\mathbb{C}}\setminus \bigcup_{i=1}^d\Delta_{\vc,i} \), \( l\in\{1,\ldots,d\} \), where we need to use the maximum modulus principle and vanishing of \( \widehat Z_{1,l+1} \) at infinity to cancel the pole of \( \Upsilon_{\n,l}^{(l)} \). This estimate immediately proves the first asymptotic formula of the theorem on closed subsets of \( \overline{\mathbb{C}}\setminus \bigcup_{j=1}^d\Delta_{\vc,j} \). Since the ratio \( \gamma_\n A_\n^{(i)}(S_\n\Phi_\n)^{(i)}/(\Pi_\n^{(i)}w_{\n,i}) \) is holomorphic outside \( \Delta_{\n,i} \), the asymptotic formula is valid on closed subsets of  \( \overline{\mathbb{C}}\setminus \Delta_{\vc,i} \) again by the maximum modulus principle for holomorphic functions. Finally, the second relation in \eqref{Y1} and \eqref{M-inverse} give us
\[
\widehat{\boldsymbol Y}_{\n\pm} = \boldsymbol C_\n^{-1} \widehat{\boldsymbol Z}_\n \boldsymbol \Pi_{\n\pm}^\mathsf{T}\boldsymbol S_{\n\pm}^{-1} \boldsymbol D_{\n\pm}^{-1} \big(\boldsymbol I\mp(1/\rho_l)\boldsymbol E_{1,l+1}\big)
\]
on any compact subset of \( \Delta_{\vc,l}^\circ \). Therefore,
\[
\big[ \widehat{\boldsymbol Y}_\n\big]_{1,l+1} = \big( 1+\mathcal O\big(|\n|^{-1}\big) \big)\frac{\Pi_{\n\pm}^{(l)}w_{\n,l\pm}}{\gamma_\n(S_\n\Phi_\n)_\pm^{(l)}} \mp (1/\rho_l)\big( 1+\mathcal O\big(|\n|^{-1}\big) \big) \frac{ \Pi_{\n\pm}^{(0)}
}{\gamma_\n(S_\n\Phi_\n)_\pm^{(0)}}.
\]
Since \(  \mp(1/\rho_l)\Pi_{\n\pm}^{(0)}/(S_\n\Phi_\n)_\pm^{(0)} = \Pi_{\n\mp}^{(l)}w_{\n,l\mp}/(S_\n\Phi_\n)_\mp^{(l)} \)  on \( \Delta_{\n,l}
\) the second asymptotic formula of the theorem now easily follows.
\end{proof}

\subsection{Recurrences}

\begin{proof}[Proof of theorem~\ref{thm:recurrence}]

It can be deduced from orthogonality relations \eqref{1.2} that
\[
R_\n^{(i)}(z) = -\frac{h_{\n,i}}{2\pi\mathrm i}\frac1{z^{n_i+1}} + \mathcal O\big(z^{-n_i-2}\big), \quad h_{\n,i} := \int P_\n(x)x^{n_i}\dd\mu_i(x),
\]
\( i\in\{1,\ldots,d\} \). In particular, we have that $m_{\vec n,i}=-2\pi\mathrm i/h_{\vec n-\vec e_i,i}$ in \eqref{Y}. Since
\[
-\frac{h_{\n,i}}{2\pi\mathrm i} = \gamma_\n\frac{1+\mathcal O\left(|\n|^{-1}\right)}{[\boldsymbol C_\n]_{i+1,i+1}}
\]
by \eqref{RniAs} and the definition of the matrix \( \boldsymbol C_\n \), we get from \eqref{por1} and \eqref{mns} that
\[
a_{\n,i} = h_{\n,i}/ h_{\n-\vec e_i,i} = \left(1+\mathcal O\left(|\n|^{-1}\right)\right)A_{\n,i},
\]
\( i\in\{1,\ldots,d\} \). Furthermore, it follows from \eqref{por}, \eqref{Ln}, and \eqref{LnInfty} that
\[
zL_{\n+\vec e_i}(z) - L_\n(z) = b_{\n,i}z^{-|\n|-1} + \mathcal O\big(z^{-|\n|-2}\big)
\]
as \( z\to\infty \), \( i\in\{1,\ldots,d\} \). Hence, we get from \eqref{sd_asl}, \eqref{Upsilons}, and \eqref{Lnl} that
\[
b_{\n,i} = \left(1+\mathcal O\left(|\n|^{-1}\right)\right)B_{\n,i},
\]
\( i\in\{1,\ldots,d\} \). As mentioned in the proof of theorem~\ref{thm:typeII}, it holds that
\[
\lim_{\mathcal N_\vc}A_{\n,i} = A_{\vc,i} \quad \text{and} \quad \lim_{\mathcal N_\vc} B_{\n,i} = B_{\vc,i},
\]
\( i\in\{1,\ldots,d\} \), from which the claim of the theorem easily follows.
\end{proof}

Nearest-neighbor recurrences \eqref{1.7} lead to other recurrence relations for multiple orthogonal polynomials \eqref{1.2}, in particular, the so-called \emph{step-line recurrence}. Given an index $n\in\N$, it can be uniquely written as $n=md+i$, $i\in\{0,\ldots,d-1\}$. Set
\begin{equation}
\label{step-line}
P_n(x) = P_{\vec i(n)}(x), \quad \text{where} \quad \vec i(n) := \big(\underbrace{m+1,\ldots,m+1}_{i\text{ times}},\underbrace{m,\ldots,m}_{d-i\text{ times}}\big), \quad |\:\vec i(n)\:|=n.
\end{equation}
It is known \cite{ApKalLLRocha06} that the polynomials $P_n(x)$ satisfy $(d+2)$-term recurrence relations
\begin{equation}
\label{sl-rec}
xP_n(x) = P_{n+1}(x) + \sum_{k=0}^d\gamma_{n,k}P_{n-k}(x).
\end{equation}
In \cite[Theorem~1.2 and Lemma~4.2]{ApKalLLRocha06} it was shown that the existence of the ratio asymptotics for the polynomials \( P_{\vec i(n)}(z) \) is equivalent to the existence of the limits for the recurrence coefficients \( \gamma_{n,k} \), which were computed for \( d=2\). With the help of theorem~\ref{thm:recurrence} we can say more.

\begin{corollary}
\label{cor:step-line}
In the setting of theorem~\ref{thm:recurrence}, let polynomials $P_n(x)$ be defined by \eqref{step-line} and $\{\gamma_{n,k}\}_{k=0}^d$ be as in \eqref{sl-rec}.  If we set $A_j:=A_{\vc,j}$ and $B_j:=B_{\vc,j}$ for $\vc=(1/d,\ldots,1/d)$, then for $i\in\{0,\ldots,d-1\}$ it holds that
\begin{equation}
\label{step-line-limit}
\lim_{m\to\infty}\left\{
\begin{array}{l}
\gamma_{md+i,0} = B_{i+1}, \medskip \\
\gamma_{md+i,1}=A_1+\cdots+A_d, \medskip \\
\gamma_{md+i,k} = \sum_{j=1}^dA_j\prod_{l=0}^{k-2}(B_j-B_{i-l}), \quad k\in\{2,\ldots,d\},
\end{array}
\right.
\end{equation}
where we understand the subindices of $B$'s cyclicly, that is, $B_{-j}=B_{d-j}$ for $j\in\{0,\ldots, d-1\}$.
\end{corollary}
\begin{proof}

Let $n=md+i$, $i\in\{0,\ldots,d-1\}$. It follows from \eqref{1.7} that
\[
zP_n(z) = P_{n+1}(z) + b_{\vec i(n),i+1}P_n(z) + \sum_{j=1}^da_{\vec i(n),j}P_{\vec i(n)-\vec e_j}.
\]
As the sum on the right-hand side of the equality above has degree at most $n-1$, the first limit in \eqref{step-line-limit} follows.  It can be inferred from \eqref{1.7} that
\begin{eqnarray}
P_{\vec i(n-l)-\vec e_j} &=& P_{\vec i(n-l)-\vec e_{i-l}} + \big(b_{\vec i(n-l)-\vec e_j-\vec e_{i-l},j} - b_{\vec i(n-l)-\vec e_j-\vec e_{i-l},i-l} \big) P_{\vec i(n-l)-\vec e_j-\vec e_{i-l}} \nonumber \\
\label{step}
&=& P_{n-l-1} + \big(b_{\vec i(n-l-1)-\vec e_j,j} - b_{\vec i(n-l-1)-\vec e_j,i-l} \big) P_{\vec i(n-l-1)-\vec e_j},
\end{eqnarray}
where we understand that \( \vec e_{i-l} = \vec e_{d+i-l} \) when \( i-l\leq 0 \). By using \eqref{step} with \( l=0 \), we get that
\[
\sum_{j=1}^da_{\vec i(n),j}P_{\vec i(n)-\vec e_j} = \sum_{j=1}^da_{\vec i(n),j} P_{n-1} + \sum_{j=1}^da_{\vec i(n),j}(b_{\vec i(n-1)-\vec e_j,j}-b_{\vec i(n-1)-\vec e_j,i}) P_{\vec i(n-1)-\vec e_j}.
\]
As the last sum above has degree at most $n-2$, the second limit in \eqref{step-line-limit} is proved. Observe that
\[
u_{n,j} := a_{\vec i(n),j}(b_{\vec i(n-1)-\vec e_j,j} - b_{\vec i(n-1)-\vec e_j,i}) \to A_j(B_j - B_i)
\]
as \( n\to\infty \). By using \eqref{step} with \( l = 1 \), we get that
\[
\sum_{j=1}^d u_{n,j}P_{\vec i(n-1)-\vec e_j} = \sum_{j=1}^d u_{n,j}P_{n-2} + \sum_{j=1}^d u_{n,j}(b_{\vec i(n-2)-\vec e_j,j}-b_{\vec i(n-2)-\vec e_j,i-1}) P_{\vec i(n-2)-\vec e_j}.
\]
As the last sum above has degree at most $n-2$, the  limit for \( \gamma_{md+i,2} \) in \eqref{step-line-limit} is established. Clearly, the rest of the limits can be easily shown by induction on \( k \).
\end{proof}

\bibliographystyle{plain}

\end{document}